\RequirePackage{fix-cm}
\documentclass[smallextended]{svjour3}       
\smartqed  
\usepackage{csquotes}
\usepackage{graphicx}
\usepackage{amsfonts}
\usepackage{amsmath}
\usepackage{amssymb}
\usepackage{geometry}
\usepackage{color}
\usepackage[hidelinks]{hyperref}
\usepackage{cite}
\newcommand{\de}{\mathrm{d}}
\begin{document}

\title{Polynomial histopolation on mock-Chebyshev segments
}

\author{Ludovico Bruni Bruno         \and
        Francesco Dell'Accio \and 
        Wolfgang Erb \and 
        Federico Nudo
}


\institute{Ludovico Bruni Bruno \at
              Department of Mathematics \enquote{Tullio Levi-Civita}, University of Padova, Italy \\
              Istituto Nazionale di Alta Matematica, Roma, Italy
              \email{ludovico.brunibruno@unipd.it}         
           \and
            Francesco Dell'Accio \at
             Department of Mathematics and Computer Science, University of Calabria, Rende (CS), Italy\\ 
             Istituto per le Applicazioni del Calcolo 'Mauro Picone', Naples Branch, C.N.R. National Research Council of Italy, Napoli, Italy
             \email{francesco.dellaccio@unical.it} 
         \and 
             Wolfgang Erb \at
              Department of Mathematics \enquote{Tullio Levi-Civita}, University of Padova, Italy \\
              \email{wolfgang.erb@unipd.it}
         \and 
            Federico Nudo \at
              Department of Mathematics \enquote{Tullio Levi-Civita}, University of Padova, Italy \\
              \email{federico.nudo@unipd.it}
}

\date{Version: July 23, 2024}

\maketitle

\begin{abstract}
In computational practice, we often encounter situations where only measurements at equally spaced points are available. Using standard polynomial interpolation in such cases can lead to highly inaccurate results due to numerical ill-conditioning of the problem. Several techniques have been developed to mitigate this issue, such as the mock-Chebyshev subset interpolation and the constrained mock-Chebyshev least-squares approximation.
The high accuracy and the numerical stability achieved by these techniques motivate us to extend these methods to histopolation, a polynomial interpolation method based on segmental function averages. 

While classical polynomial interpolation relies on function evaluations at specific nodes, histopolation leverages averages of the function over subintervals. In this work, we introduce three types of mock-Chebyshev approaches for segmental interpolation and theoretically analyse the stability of their Lebesgue constants, which measure the numerical conditioning of the histopolation problem under small perturbations of the segments. We demonstrate that these segmental mock-Chebyshev approaches yield a quasi-optimal logarithmic growth of the Lebesgue constant in relevant scenarios. Additionally, we compare the performance of these new approximation techniques through various numerical experiments.
\keywords{Polynomial interpolation on segments \and histopolation \and mock-Chebyshev segments \and constrained mock-Chebyshev least-squares approximation \and stability of the Lebesgue constant under perturbations}
\subclass{33F05 \and 41A05 \and 41A10}
\end{abstract}

\section{Introduction}
\label{intro}
Univariate polynomial interpolation is a classical numerical method for approximating a function $f$ over a specified interval $[a,b]$ from a set of function samples. More precisely, assuming to know the evaluations of a function $f$ on a grid of $n$ points
\begin{equation*}
    X_n=\left\{x_1,\dots,x_n\right\}, \qquad a\le x_1 < \dots < x_{n} \le b, 
\end{equation*}
the idea is to approximate the function $f$ with the unique polynomial $p_{n-1}$ in the space $\mathbb{P}_{n-1}$ of polynomials of degree less or equal to $n-1$ satisfying the interpolation conditions
\begin{equation*}
    p_{n-1}(x_i)=f(x_i), \qquad i=1,\dots, n.
\end{equation*}
Assuming that the function $f$ is sufficiently smooth, meaning it belongs to the space $C^{n-1}([a, b]) $ of all continuously differentiable functions with continuous first $n-1$ derivatives on $[a, b]$, and that $ f^{(n)}(x) $ exists at each point $x\in(a, b)$, the placement of the nodes $X_n$ affects the accuracy of the interpolation via the well-known remainder formula
\begin{equation} \label{eq:remainder}
f(x) - p_{n-1}(x) = \frac{\prod\limits_{i=1}^n (x-x_i)}{n!} f^{(n)} (\xi) ,
\end{equation}
see e.g.~\cite[Thm. $ 3.1.1 $]{Davis:1975:IAA}. The placement of the point $ \xi $ depends on the set of nodes $ X_n $ and the function $ f $ itself. The error formula in Eq. \eqref{eq:remainder} evidences the relevance of a proper selection of the node set $X_n$. This gets particularly evident for the interpolation of functions with fast growing derivatives, as for instance the Runge function $1/(1-25 x^2)$ in \cite{Runge1901}.

\subsection{Histopolation}

The concept of polynomial interpolation can be generalized to more abstract settings in which the given information does not only consist of function evaluations but of general functionals~\cite{Rivlin}. For this, let $ \{ \mu_1,\dots,\mu_n \} $ be $n$ linear functionals acting on a function $f$ and denote by 
\begin{equation} \label{eq:gramian}
V = 
\begin{bmatrix}
\mu_{1}(u_1) & \mu_{1}(u_2) & \cdots & \mu_{1}(u_{n})\\
\mu_{2}(u_1) & \mu_{2}(u_2) & \cdots & \mu_{2}(u_{n})\\
\vdots  & \vdots  & \ddots & \vdots  \\
\mu_{n}(u_1) & \mu_{n}(u_2) & \cdots & \mu_{n}(u_{n})\\
\end{bmatrix}
\end{equation}
the Gramian of these functionals with respect to a  polynomial basis $ \{u_1, \ldots, u_{n} \} $ of the space $ \mathbb{P}_{n-1} $. In this more abstract view, if $ \det V \ne 0 $, and by knowing the data
\begin{equation*}
    \mu_i(f), \qquad i=1,\dots,n,
\end{equation*}
interpolation becomes the process of identifying an element $p_{n-1}\in \mathbb{P}_{n-1}$ based on the interpolation conditions
\begin{equation} \label{eq:histopolatingconditions}
\mu_i(p_{n-1})=\mu_i(f), \qquad i=1,\dots,n.
\end{equation}
If $ \mu_i $ is the classical evaluation functional on the $i$-th node $ x_i \in X_n $, $ V $ corresponds to the well-known Vandermonde matrix. If the functionals $ \mu_i $ correspond to integrals over segments $ s_i $, the respective interpolation is referred to as \emph{histopolation}~\cite{Robidoux}, and plays a relevant role in splines~\cite{Schoenberg}, preconditioning~\cite{HiptmairXu} and conservation of physical quantities~\cite{HiemstraJCP}. 
Formally, we can define the linear functionals in histopolation by the mean values
\begin{equation}\label{linfunhist}
\mu_i(f):= \frac{1}{|s_i|}\int_{s_i}f(x) \de x, \qquad i=1,\dots,n,
\end{equation}
where $s_i$, $i=1,\dots,n$, are subintervals in $[a, b]$. A set of segments $\mathcal{S} := \{s_1, \dots, s_n\}$, such that the matrix  $V$ defined in~\eqref{eq:gramian}, related to $\mathbb{P}_{n-1}$ and the linear functionals~\eqref{linfunhist}, is non-singular, is called \textit{unisolvent}.
Histopolation leverages information about the integral or the mean value of the function $f$ over segments within $[a, b]$. This allows to approximate larger families of functions, even functions with jumps and discontinuities.  The interpolation on segments only demands that the function is essentially bounded, a significantly less restrictive condition compared to the continuity needed for classical polynomial interpolation. Nevertheless, if $ f $ is sufficiently regular and the segments $\mathcal{S}$ are non-overlapping, the error formula \eqref{eq:remainder} can be extended also to histopolation problems.
\begin{proposition} \label{prop-1}
    Let $ f \in C^{n-1} ([a,b]) $ and assume that $ f^{(n)} (x) $ exists at each point $ x \in (a,b)$. Let $ \mathcal{S} := \{ s_1, \ldots, s_n \} $ be a collection of segments such that $ | s_i \cap s_j | = 0 $ if $ i \ne j $. If $p_{n-1}\in\mathbb{P}_{n-1}$ is the unique interpolating polynomial satisfying
    \begin{equation} \label{eq:explicithistopolation}
        \frac{1}{|s_i|}\int_{s_i} f (x) \de x = \frac{1}{|s_i|} \int_{s_i} p_{n-1} (x) \de x, \qquad i=1,\dots,n,
    \end{equation} 
    then there exist $ \Bar{\xi}, \xi_{1}, \ldots, \xi_n \in [a,b] $ such that
    \begin{equation} \label{eq:remainderhistpolation}
        f(x) - p_{n-1}(x) = \frac{\prod\limits_{i=1}^n (x-\xi_i)}{n!} f^{(n)} (\Bar{\xi}) .
    \end{equation}   
\end{proposition}

\begin{proof}
The conditions $ | s_i \cap s_j | = 0 $ for $ i \ne j $ guarantee the unisolvence of the segment set $ \mathcal{S} $ (see \cite[Prop. 3.1]{Bruno:2023:PIO}), and thus the existence of a unique interpolating polynomial $p_{n-1} \in \mathbb{P}_{n-1}$. To obtain the remainder formula \eqref{eq:remainderhistpolation}, we proceed as follows. We rewrite the set of equations \eqref{eq:explicithistopolation} as
\begin{equation*}
\frac{1}{|s_i|} \int_{s_i} \left(f(x)-p_{n-1}(x)\right) \de x = 0, \qquad i=1,\dots,n .
\end{equation*}
Then, by the mean value theorem, for each $ i = 1, \ldots, n,$ there exists $ \xi_i \in s_i $ such that 
\begin{equation*}
f(\xi_i) - p_{n-1}(\xi_i) = 0, \qquad i=1,\dots,n.    
\end{equation*}
Notice that, for a fixed $ x \notin\left\{\xi_1, \ldots,  \xi_n\right\}$, we can define the function
\begin{equation}\label{eq2prv}
    K(x) := \frac{f(x) - p_{n-1} (x)}{\prod\limits_{i=1}^n (x - \xi_i)}
\end{equation}
and consider the following function of $\xi$
\begin{equation} \label{eq:usefulprop1}
W(\xi) := f(\xi) - p_{n-1} (\xi) - \prod_{i=1}^n (\xi - \xi_i) K(x) .
\end{equation}
By construction, $ W(\xi) $ vanishes at the $n+1$ points $ x, \xi_1, \ldots,  \xi_n $, which are distinct since $ |s_i \cap s_j | = 0 $, that is, segments in $ \mathcal{S} $ intersect at most in their vertices. By the generalized Rolle's Theorem~\cite[Thm. $ 1.6.3 $]{Davis:1975:IAA}, it follows that $ W^{(n)} (\Bar{\xi}) = 0 $ for some $ \Bar{\xi} $ such that $ \min(\xi,\xi_1) < \Bar{\xi} < \max(\xi,\xi_n) $. Thus, differentiating Eq. \eqref{eq:usefulprop1} $n$ times with respect to $ \xi $, and evaluating the expression at $ \Bar{\xi} $, we obtain
\begin{equation}\label{eq1prv}
K(x)=\frac{f^{(n)} (\Bar{\xi})}{n!}. 
\end{equation}
By combining~\eqref{eq2prv} and~\eqref{eq1prv}, we get the statement of the theorem. If $ x \in\left\{\xi_1, \ldots,  \xi_n\right\}$, then Eq.~\eqref{eq:remainderhistpolation} is trivially satisfied. \qed
\end{proof}

The larger amount of free parameters in histopolation makes the identification of suitable segments for the interpolation and approximation of functions $f$ a slightly more challenging task than in classical nodal interpolation. To reduce the number of free parameters, three distinct relevant cases were identified and studied in detail in~\cite{Bruno:2023:PIO}:
\begin{itemize}
    \item [(C1)] \textit{Chains of intervals}: for $n+1$ nodes $a = x_0 < x_1 < \cdots < x_{n-1} < x_n = b$ the $n$ segments in $\mathcal{S}$ are given by
    \begin{equation*}
        s_i=[x_{i-1},x_i], \qquad i=1,\dots, n.
    \end{equation*}
    In this case, we have
    \begin{equation*}
        [a,b]=\bigcup_{i=1}^n s_i.
    \end{equation*}
    
    \item[(C2)] \textit{Segments with uniform arc-length}: the segments $s_i$ in the interval $I = [-1,1]$ are given by
\begin{equation*}
    s_i=[\cos(\tau_i+\rho), \cos(\tau_i-\rho)], \qquad i=1,\dots,n,
\end{equation*}
with $0<\tau_1<\dots<\tau_n<\pi$ and $0<\rho<\pi$.
    \item[(C3)] \textit{Segments with identical left endpoints}: where the segments are given by
    \begin{equation*}
        s_i=[a,x_i], \qquad i=1,\dots,n,
    \end{equation*}
    with the fixed left endpoint $a$ and the right endpoints given by an increasing sequence of points $a<x_1 < \ldots < x_n$. 
\end{itemize}
For the described scenarios, Chebyshev distributions for the end and mid-points, as well as Fekete-type techniques to determine the free parameters turned out to be surprisingly effective approaches to identify quasi-optimal segments for histopolation \cite{Bruno:2023:PIO,BruniErbFekete}. A case of great practical interest arises when the segments are based on a set of equispaced end-points. However, similar to classical polynomial interpolation, the histopolation of a function $f$ based on equispaced segments can lead to ill-conditioning in the calculation of the interpolating polynomial as soon as $n$ gets sufficiently large. This ill-conditioning in the equispaced setting can be visualized with a highly oscillatory error phenomenon close to the end-points of the interval referred to as Runge phenomenon. For segmental interpolation, it was studied numerically and analytically in \cite{BruniRunge,Bruno:2023:PIO} in terms of the Lebesgue constant, a measure for the numerical conditioning of an interpolation problem. There exist several possible strategies to mitigate ill-conditioning. In classical polynomial interpolation, some of these techniques have been, for instance, proposed in \cite{Boyd:2009:DRP,DeMarchi:2015:OTC,DeMarchi:2020:PIV,DellAccio:2022:GOT}.

The main goal of this paper is to generalize the mock-Chebyshev subset interpolation~\cite{Boyd:2009:DRP} and the constrained mock-Chebyshev least-squares approximation~\cite{DeMarchi:2015:OTC,DellAccio:2022:GOT} to histopolation problems, introducing three novel methods, specifically tailored for polynomial approximation based on average data of functions over equispaced segments. In defining these methods, we assume to work, without loss of generality, with the reference interval $I=[-1,1]$. The key idea of these methods is to reduce the number of available data to segments with a Chebyshev-type distribution. For this, we will theoretically analyse in Section \ref{section-stability} what happens if a numerical well-conditioned set of segments is perturbed and how well the growth of the respective Lebesgue constants is maintained. A detailed analysis of the introduced segmental mock-Chebyshev methods is then the core of Section~\ref{Sec2}. We will prove that two of these three new methods achieve quasi-optimal logarithmic growth rates for their Lebesgue constant. The numerical results presented in Section~\ref{SecNum} conclude this article and demonstrate the accuracy of the proposed methods. 

\section{Stability of the segmental Lebesgue constant} \label{section-stability}

\subsection{Numerical conditioning and the Lebesgue constant}

In practice, higher-order derivatives of a function $f$ are typically not available and the nodes $ \xi $ are not explicitly known. The error formulas in  Eqs. \eqref{eq:remainder} and \eqref{eq:remainderhistpolation} are therefore of limited practical value in assessing the quality of a particular choice of nodes or segments for function approximation. In the context of nodal interpolation, this task is usually given to the Lebesgue constant \cite[Eq. 4.1.9]{Rivlin}. For histopolation, an analog quantity can be used for this task \cite{Bruno:2023:PIO},  and more generally also for differential forms \cite{ARRLeb}. For this, let $ \mathcal{S} = \{ s_1, \ldots, s_n \} $ be a unisolvent set of segments for the space $ \mathbb{P}_{n-1} $.
For this set, we consider the respective Lagrange basis $ \{ \ell_{s_1}, \ldots \ell_{s_n} \} $ uniquely determined by the conditions
\begin{equation} \label{eq:lagrangefunctions}
    \frac{1}{|s_i|} \int_{s_j} \ell_{s_i}(x) \mathrm{d}x = \delta_{ij},
\end{equation}
where $\delta_{ij}$ is the Kronecker delta symbol.
With this at hand, the segmental Lebesgue constant is defined as
\begin{equation} \label{eq:LebSegm}
\Lambda_{n} (\mathcal{S}) := \sup_{s \subset I} \frac{1}{\left|s\right|} \sum_{i = 1}^n \left\vert \int_{s} \ell_{s_i} (x) \de x \right\vert = \max_{x\in I} \sum_{i=1}^n \left\lvert \ell_{s_i} (x) \right\rvert, 
\end{equation}
the last equality being granted by the mean value theorem \cite[Eq. (19)]{BruniErbFekete}.
This quantity measures the numerical conditioning of the generalized interpolation problem \cite[Sect. $ 7 $]{ABR22}, and it can be proven that under the assumptions of Proposition \ref{prop-1} it coincides with the operator norm (induced by the uniform norm) of the interpolation operator that maps a function $f$ onto the polynomial $p_{n-1}$ based on the conditions \eqref{eq:histopolatingconditions}, see \cite{Bruno:2023:PIO}. A thorough description of this quantity involving the above scenarios (C1), (C2) and (C3) is provided in \cite{Bruno:2023:PIO,BruniErbFekete}, where several interconnections between \eqref{eq:LebSegm} and the nodal Lebesgue constant have been shown. Moreover, if the segments shrink to a set of disjoint nodes, the Lebesgue constant $\Lambda_n(\mathcal{S})$ tends to the classical Lebesgue constant for nodal polynomial interpolation \cite{BruniErbFekete}. 

\subsection{Stability of segmental norming sets}

In the following, we use $\|f\| =\sup _{x \in I}|f(x)|$ to denote the uniform norm of a bounded function $f$ on the reference interval $I = [-1,1]$ and
\[\|f\|_{\mathcal{S}} =\sup_{s \in \mathcal{S}} \frac{1}{|s|} \left| \int_s f(x) \mathrm{d}x \right| \]
to denote the respective discrete counterpart related to the average information of the function $f$ on the segments $\mathcal{S}$. We call a set $\mathcal{S}$ a \emph{segmental norming set} for the space $\mathbb{P}_m$ of all polynomials of degree $m$ if there exists a norming constant $\lambda = \lambda(\mathcal{S},\mathbb{P}_m)$ such that the inequality 
$$
\|p\| \leq \lambda \|p\|_{\mathcal{S}}
$$
holds true for all $p \in \mathbb{P}_m$. For such norming sets, we have the following stability result. It is a direct generalization of a respective result \cite[Prop. $ 1 $]{PiazzonVianello} for nodal norming sets. 

\begin{proposition}[Stability related to segmental norming sets] \label{prop-main}
Assume that the set $\mathcal{S} = \{s_1, \ldots, s_n\}$ is a segmental norming set for $\mathbb{P}_m$ with norming constant $\lambda = \lambda(\mathcal{S},\Pi_m)$. Further, let $\mathcal{\tilde{S}} = \{\tilde{s}_1, \ldots, \tilde{s}_n\}$ consist of perturbations $\tilde{s}_i = [\tilde{\alpha}_i,\tilde{\beta}_i]$ of the segments $s_i = [\alpha_i,\beta_i]$ in $\mathcal{S}$ such that for $\varepsilon > 0$ we have
\[ |\alpha_i - \tilde{\alpha}_i| \leq \varepsilon \;\; \text{and} \;\; |\beta_i - \tilde{\beta}_i| \leq \varepsilon  \quad \text{for all $i = 1, \ldots, n$}.\]

If $\varepsilon$ is small enough in the sense that $\varepsilon=\frac{\alpha}{\lambda m^2}$ for some $0 < \alpha < 1$, then the following inequality holds true:
$$
\|p\| \leq \frac{\lambda}{1-\alpha}\|p\|_{\mathcal{\tilde{S}}} \quad \text{for all}\; p \in \mathbb{P}_m,
$$
i.e., also $\mathcal{\tilde{S}}$ is a segmental norming set for $\mathbb{P}_m$ with norming constant $\lambda/(1 - \alpha)$. 
\end{proposition}

\begin{proof} Let $p \in \mathbb{P}_m$ and $x, \tilde{x} \in I$ such that $|x - \tilde{x}| \leq \epsilon$. Then, Markov's inequality for algebraic polynomials of degree $m$ on $I$ provides the estimate
\begin{equation} \label{eq:prop1}
|p(x) - p(\widetilde{x})| = \left|\int_{\tilde{x}}^{x} p'(t) \mathrm{d} t\right| \leq 
\int_{\tilde{x}}^{x} |p'(t)| \mathrm{d} t 
\leq \varepsilon \sup_{x \in I}|p'(x)| \leq \varepsilon m^2 \|p\|.
\end{equation}

We consider now an arbitrary polynomial $p \in \mathbb{P}_m$. For this polynomial $p$, there exists a segment $s_i \in \mathcal{S}$ such that $\frac{1}{|s_i|} | \int_{s_i} p(x) \mathrm{d} x | =\|p\|_{\mathcal{S}}$. By the definition of the perturbed set $\mathcal{\tilde{S}}$ there exists also an affine linear and monotonically increasing function $\phi: \mathbb{R} \to \mathbb{R}$ such that $\phi(s_i) = \tilde{s}_i$ and $|x - \phi(x)| \leq \varepsilon$ for all $x \in s_i$. Using the estimate in \eqref{eq:prop1} in combination with trivial bounds and inequalities for the involved integrals, we get the bound
\begin{align*}\frac{1}{|s_i|} \left| \int_{s_i} p(x) \mathrm{d} x \right| & 
\leq \frac{1}{|s_i|} \left| \int_{s_i} p(\phi(x)) \mathrm{d} x \right| + \frac{1}{|s_i|} \left| \int_{s_i} p(x) - p(\phi(x)) \mathrm{d} x \right| \\
& \leq \frac{1}{|\tilde{s}_i|} \left| \int_{\tilde{s}_i} p(x) \mathrm{d} x \right| + \frac{1}{|s_i|} \int_{s_i} |p(x) - p(\phi(x))| \mathrm{d} x  \\
& \leq \frac{1}{|\tilde{s}_i|} \left| \int_{\tilde{s}_i} p(x) \mathrm{d} x \right| + \sup_{x \in s_i} |p(x) - p(\phi(x))| \\
& \leq \|p\|_{\mathcal{\tilde{S}}} + \varepsilon m^2 \|p\|.
\end{align*} 

Now, using this bound and the fact that $\mathcal{S}$ is a segmental norming set for $\mathbb{P}_m$ we can conclude 
$$
\|p\| \leq \lambda \|p\|_{\mathcal{S}} \leq \lambda \|p\|_{\mathcal{\tilde{S}}} + \varepsilon m^2 \lambda \|p\|.
$$
Therefore, since $\varepsilon=\frac{\alpha}{\lambda m^2}$ with $0<\alpha<1$, we get the statement of the proposition. \qed
\end{proof}

\begin{remark} \label{rem-1}
The statement of Proposition \ref{prop-main} remains consistently true if the segment set $\mathcal{S}$ is replaced by a node set (i.e., if the intervals $s_i$ shrink to a single point) and the elements in $\mathcal{\tilde{S}}$ either consist of small segments or nodes in $\varepsilon$-balls around these nodes. In this case, in the proof of Proposition \ref{prop-main} the averages $\frac{1}{|s_i|} \int_{s_i} p(x) \mathrm{d} x $ have to be replaced by point evaluations. This is consistent with the fact that nodal polynomial interpolations can be seen as a limit case of segmental polynomial interpolation in which segments shrink to a single node, as discussed in \cite{BruniErbFekete}. In this sense, Proposition \ref{prop-main} can be regarded as an extension of the nodal stability result \cite[Prop. $ 1 $]{PiazzonVianello} to norming sets that contain point evaluations and function averages over segments. 
\end{remark}

Proposition \ref{prop-main} provides directly the following result on the stability of the Lebesgue constants.

\begin{theorem}[Stability of the segmental Lebesgue constant] \label{thm-main}
Assume that $\mathcal{S} = \{s_1, \ldots, s_n\}$ is a unisolvent set of segments for $\mathbb{P}_{n-1}$ with the Lebesgue constant $\Lambda_n(\mathcal{S})$. Further, let $\mathcal{\tilde{S}} = \{\tilde{s}_1, \ldots, \tilde{s}_n\}$ be a perturbation of $\mathcal{S}$ such that $\tilde{s}_i = [\tilde{\alpha}_i,\tilde{\beta}_i]$ and $s_i = [\alpha_i,\beta_i]$ satisfy
\[ |\alpha_i - \tilde{\alpha}_i| \leq \varepsilon \;\; \text{and} \;\; |\beta_i - \tilde{\beta}_i| \leq \varepsilon  \quad \text{for all $i = 1, \ldots, n$ and $\varepsilon > 0$}.\]
If $\varepsilon \leq \frac{\alpha}{\Lambda_n(\mathcal{S}) (n-1)^2}$ for some $0 < \alpha < 1$, then also $\mathcal{\tilde{S}}$ is unisolvent for $\mathbb{P}_{n-1}$ and the unique operator $\Pi_n(\mathcal{\tilde{S}})$ mapping a function to its segmental polynomial interpolant in $\mathbb{P}_{n-1}$ satisfies
$$
\| \Pi_n(\mathcal{\tilde{S}}) \|_{\mathrm{op}} \leq \frac{1}{1-\alpha} \Lambda_n(\mathcal{S}).
$$
Note that the operator norm $ \| \Pi_n(\mathcal{\tilde{S}}) \|_{\mathrm{op}}$ with respect to the uniform norm corresponds to the segmental Lebesgue constant $\Lambda_n(\mathcal{\tilde{S}})$ if the segments in $\mathcal{\tilde{S}}$ are non-overlapping in the sense described in the assumptions of Proposition \ref{prop-1}.
\end{theorem}
\begin{proof}
In the setting of segmental polynomial interpolation, the Lebesgue constant $\Lambda_n(\mathcal{S})$ corresponds to a norming constant $\lambda(\mathcal{S},\mathbb{P}_{n-1})$ for the space $\mathbb{P}_{n-1}$. Therefore, Proposition \ref{prop-main} states that also a perturbed set $\mathcal{\tilde{S}}$ of segments with $\varepsilon$ satisfying $\varepsilon \leq \frac{\alpha}{\Lambda_n(\mathcal{S}) (n-1)^2}$ provides a norming set for $\mathbb{P}_{n-1}$. This implies that the linear mapping
\[ \mathbb{P}_{n-1} \to \mathbb{R}^n,\quad p \mapsto \left( \frac{1}{|\tilde{s}_1|} \int_{\tilde{s}_1} p(x) \mathrm{d} x, \ldots, \frac{1}{|\tilde{s}_n|} \int_{\tilde{s}_n} p(x) \mathrm{d} x \right)\]
is bijective, and therefore, that $\mathcal{\tilde{S}}$ is unisolvent for $\mathbb{P}_{n-1}$. In addition, the bound of the norming constant in Proposition \ref{prop-main} provides the bound $\| \Pi_n(\mathcal{\tilde{S}}) \|_{\mathrm{op}} \leq \frac{1}{1-\alpha} \Lambda_n(\mathcal{S})$ for the inverse of this mapping. It is shown in \cite{BruniErbFekete} that the operator norm $\| \Pi_n(\mathcal{\tilde{S}}) \|_{\mathrm{op}}$ corresponds to the Lebesgue constant $\Lambda_n(\mathcal{\tilde{S}})$ of the set $\mathcal{\tilde{S}}$ if the segments $\tilde{s}_i$ in $\mathcal{\tilde{S}}$ overlap at most in their endpoints. \qed
\end{proof}

\begin{remark} \label{rem-2}
Similarly as discussed in Remark \ref{rem-2}, the statement of Proposition \ref{thm-main} remains also true if the segment set $\mathcal{S}$ is replaced by a point set. In this case, the elements of $\mathcal{\tilde{S}}$ are small segments (or nodes) close to this initial point set and the operator norm $\| \Pi_n(\mathcal{\tilde{S}}) \|_{\mathrm{op}}$ can bounded in terms of a nodal Lebesgue constant.  
\end{remark}

\section{Polynomial interpolation on mock-Chebyshev segments} \label{Sec2}
On the reference interval $I=[-1,1]$, we consider the equispaced grid of $n+1$ nodes given as 
\begin{equation}\label{eqGrid}
    X^{\mathrm{eq}}_{n+1}=\left\{x_i=-1+\frac{2}{n}i\, :\, i=0,\dots,n\right\}
\end{equation}
and the respective set of uniform segments
\begin{equation}\label{setsegmeq}
    \mathcal{S}_{n}^{\mathrm{eq}}=\left\{s_1,\dots,s_n\right\},
\end{equation}
where the equispaced segments
\begin{equation}\label{segseq}
    s_i=[x_{i-1},x_{i}], \qquad i=1,\dots,n,
\end{equation}
contain the nodes of $X^{\mathrm{eq}}_{n+1}$ as endpoints. 
Assuming to know the average values
\begin{equation}\label{hypoth}
\mu_i(f):=\frac{2}{n}\int_{x_{i-1}}^{x_{i}}f(x) \de x, \qquad i=1,\dots,n,
\end{equation}
of a bounded function $f : I \to \mathbb{R}$, our goal is to determine a real-valued polynomial $p$ that approximates the available data in the sense that
\begin{equation*}
\mu_i(p)\approx \mu_i(f), \qquad i=1,\dots, n.
\end{equation*}
As a first attempt, one could calculate a polynomial histopolant in the space $\mathbb{P}_{n-1}$ of polynomials of degree at most $n-1$ using the averages over all segments in $\mathcal{S}_{n}^{\mathrm{eq}}$. This provides an interpolating polynomial 
 \begin{equation*}
     p^{\mathrm{eq}}_{n-1}\in\mathbb{P}_{n-1}
 \end{equation*}
that satisfies the histopolation conditions
\begin{equation*}
    \mu_i\left( p^{\mathrm{eq}}_{n-1}\right)=\mu_i(f), \qquad i=1,\dots,n.
\end{equation*}
However, as shown in~\cite{Bruno:2023:PIO}, this approach does not produce satisfactory results due to the ill-conditioned nature of the interpolation problem on uniform segments. 

\vspace{2mm}

In this paper, we aim to extend the mock-Chebyshev subset interpolation method and the constrained mock-Chebyshev least-squares approximation method, described in~\cite{Boyd:2009:DRP,DeMarchi:2015:OTC,DellAccio:2022:GOT}, to the case of histopolation and under the assumption of knowing the integrals of the function $f$ over the segments of the set $\mathcal{S}_{n}^{\mathrm{eq}}$ in \eqref{setsegmeq}. Specifically, we introduce three distinct methods to adapt the advantageous properties of the mock-Chebyshev approach for an accurate and numerically stable approximation of functions. In particular, these approximation techniques aim to mitigate the Runge phenomenon in histopolation, which can occur if a function is interpolated with a polynomial of large degree $n-1$ using $n$ equispaced segments based on the endpoints in $X_{n+1}^{\mathrm{eq}}$.

Nodal mock-Chebyshev subset interpolation uses nodal evaluations of a (continuous) function $f$ on a proper subset of the  equispaced set $X_{n+1}^{\mathrm{eq}}$ in~\eqref{eqGrid} to generate a low-order polynomial interpolant of the function $f$. More specifically, a proper subset of $m+1$ nodes
\begin{equation}\label{mock-Chebnods}
    X_{m+1}^{\mathrm{MC}}=\left\{x^{\mathrm{MC}}_0,\dots, x^{\mathrm{MC}}_m\right\},
\end{equation}
referred to as \textit{mock-Chebyshev} nodes~\cite{Boyd:2009:DRP,Ibrahimoglu:2020:AFA,Ibrahimoglu:2024:ANF}, 
is extracted from the equispaced set $X_{n+1}^{\mathrm{eq}}$. These nodes are chosen to mimic the behavior of the Chebyshev-Lobatto nodes as closely as possible. Mathematically, this is achieved by solving the minimization problem
\begin{equation}\label{Probmock}
\min_{k=0,\dots,n}\left\lvert x_k-x^{\mathrm{CL}}_i \right\rvert^2, \qquad i=0,\dots,m,
\end{equation}
where 
\begin{equation}\label{ChebLobNodes}
    x_i^{\mathrm{CL}}= - \cos\left(\nu_{m-i}\right), \qquad \nu_i=\frac{\pi}{m}i, \qquad i=0,\dots,m,
\end{equation}
denote the Chebyshev-Lobatto nodes $X^{\mathrm{CL}}_{m+1}$ in $I$. 
A maximal number $m$ can be determined such that the nodes solving the minimization problem are distinct. This value was computed in~\cite{Boyd:2009:DRP} and it is equal to
\begin{equation}   \label{valueofm}
    m=\left\lfloor \pi \sqrt{\frac{n}{2}} \right\rfloor.
\end{equation}
This identity is optimal in the sense that it provides the largest possible value $m$ to still identify $ m $ different mock-Chebyshev nodes. Once the number $m$ is fixed, we can thus increase the value $ n $ by any larger natural number to control the distance between a Chebyshev node and a mock-Chebyshev one. We may formalise this fact in the following way (see also \cite{DeMarchi:2015:OTC}).
\begin{lemma} \label{lem:closesegments}
    Let $ \varepsilon > 0 $ and $m \in \mathbb{N}$ be fixed. Choosing $n \in \mathbb{N}$ such that $$ n \geq \max \left \{  m^2 \frac{2}{\pi^2} , \, \frac{1}{\epsilon}\right \} $$ we get unique mock-Chebyshev nodes $x_i^{\mathrm{MC}}$ such that $ \left\lvert x_i^{\mathrm{CL}} - x_i^{\mathrm{MC}} \right\rvert \leq \varepsilon $ for all $ i=0,\dots,m$.
\end{lemma}

\begin{proof}
By \eqref{valueofm}, the condition $n \geq m^2 \frac{2}{\pi^2}$ guarantees the unique extractability of the mock-Chebyshev points from the equispaced grid $X_{n+1}^{\mathrm{eq}}$(cf. \cite{Boyd:2009:DRP}). Further, the maximal distance of a point $x \in I$ from a point $X_{n+1}^{\mathrm{eq}}$ is given by $1/n$. This proves the claim of the lemma. \qed
\end{proof}

In mock-Chebyshev subset interpolation, the nodes from the complementary set
\begin{equation}\label{compl}
   X^{\prime}=X_{n+1}^{\mathrm{eq}}\setminus X_{m+1}^{\mathrm{MC}}
\end{equation}
are not used. A natural extension of mock-Chebyshev interpolation is therefore a respective regression approach. The idea of the constrained mock-Chebyshev least-squares approximation is to approximate the function $f$ with a polynomial of degree $r>m$ obtained by interpolating the function $f$ on the mock-Chebyshev nodes, and using the nodes of $X^{\prime}$ to improve the accuracy of the approximation through a simultaneous regression~\cite{DeMarchi:2015:OTC,DellAccio:2022:GOT}. This technique has been successfully applied in various applications~\cite{DellAccio:2022:CMC,DellAccio:2022:AAA,DellAccio:2023:PIR,DellAccio:2024:PAO,DellAccio:2024:NAO, DellAccio:2024:AEO}.

In the following, we introduce now three new approximation methods designed for polynomial approximation based on integral information on subsets of equispaced segments. We will refer to them as the \emph{concatenated segmental mock-Chebyshev method}, the \emph{quasi-nodal segmental mock-Chebyshev method}, and the \emph{constrained segmental mock-Chebyshev method}.

\subsection{Concatenated segmental mock-Chebyshev method}
\label{sub1}
The key idea of this method is to subdivide the interval $[-1,1]$ into segments whose endpoints are the mock-Chebyshev nodes~\eqref{mock-Chebnods}. By exploiting the linearity of the integral and the available data~\eqref{hypoth}, we can perform the interpolation over the segments of this partition, see Fig.~\ref{prbseg}. This approach aims to harness the quasi-optimality of the Chebyshev--Lobatto nodes as endpoints for segments in segmental polynomial interpolation \cite{Bruno:2023:PIO}. For this, we denote by $ \mathcal{S}_m^{\mathrm{CL}} :=\left\{ s_1^{\mathrm{CL}}, \ldots, s_m^{\mathrm{CL}}\right\}$ the set of Chebyshev segments 
$s_i^{\mathrm{CL}} = [x_{i-1}^{\mathrm{CL}},x_i^{\mathrm{CL}}]$, $i = 1, \ldots, m$ with the Chebyshev--Lobatto nodes $x_i^{\mathrm{CL}}$ as endpoints. Approximating the Chebyshev--Lobatto nodes $x_i^{\mathrm{CL}}$ with the mock-Chebyshev nodes $x_{i}^{\mathrm{MC}}$ we get the \textit{mock-Chebyshev segments}
\begin{equation} \label{eq:MCsegments}
    s^{\mathrm{MC}}_i=\left[x_{i-1}^{\mathrm{MC}},x_{i}^{\mathrm{MC}}\right], \qquad i=1,\dots,m,
\end{equation}
as approximations of the Chebyshev segments $s_i^{\mathrm{CL}}$. For these mock-Chebyshev segments, 
we denote by
\begin{equation*}
\mu_{i}^{\mathrm{MC}}(f)= \frac{1}{|x_{i}^{\mathrm{MC}} - x_{i-1}^{\mathrm{MC}}|}\int_{x_{j-1}^{\mathrm{MC}}}^{x_{i}^{\mathrm{MC}}} f(x) \de x, \qquad i=1,\dots,m,
\end{equation*}
the averages of the function $f$ over the set of segments
\begin{equation*}
    \mathcal{S}_m^{\mathrm{MC}}=\left\{s_1^{\mathrm{MC}},\dots,s_m^{\mathrm{MC}}\right\}.
\end{equation*}
Exploiting the linearity of the integral and the uniformity of the segments in $\mathcal{S}_{n}^{\mathrm{eq}}$, we observe that
\begin{equation}\label{relintegrals}
\mu_i^{\mathrm{MC}}(f)= \frac{1}{k_i}\sum_{\ell=\iota_i+1}^{\iota_i+k_i}\mu_{\ell}(f), \qquad i=1,\dots,m,
\end{equation}
where $\mu_{\ell}$ is defined in~\eqref{hypoth} and the indices $\iota_i$ and $k_i$ are uniquely determined by
\begin{equation*}
x_{\iota_j}=x_{j-1}^{\mathrm{MC}}, \qquad x_{\iota_j+k_j}=x_{j}^{\mathrm{MC}}.
\end{equation*}
With this, we further have
\begin{equation*}
 x^{\mathrm{MC}}_{j-1}=x_{\iota_{j}}<x_{\iota_{j}+1}<\cdots<x_{\iota_{j}+k_j}=x^{\mathrm{MC}}_{j}
\end{equation*}
and
\begin{equation}\label{segconc}
 s_j^{\mathrm{MC}}= s_{\iota_j+1}\cup s_{\iota_j+2}\cup \dots \cup s_{\iota_j+k_j}.
\end{equation}
The set of segments $\mathcal{S}_m^{\mathrm{MC}}$ belongs to the class (C1) of concatenated segments. Further, as the segments $\mathcal{S}_m^{\mathrm{MC}}$ approximate the Chebyshev segments $\mathcal{S}_m^{\mathrm{CL}}$, they provide an approximation to segments lying in the intersection of the classes (C1) and (C2), both discussed in \cite[Example $1.1$]{Bruno:2023:PIO}. For the class (C1), a general result on the unisolvence of the polynomial interpolation is provided in \cite[Prop. 3.1]{Bruno:2023:PIO}. 
Although the approximation of the mock-Chebyshev segments is based on a concatenation of uniform segments, the Lebesgue constants associated with the segments $ \mathcal{S}_m^{\mathrm{MC}} $ display a slow logarithmic growth. This logarithmic trend of the Lebesgue constant can be seen numerically, cf. Fig. \ref{fig:LebCheb}, but also shown analytically under particular assumptions.  

\begin{theorem} \label{thm:stabilityconcatenated}
For $m \in \mathbb{N}$, let 
$$ n \geq \frac{c^{\mathrm{CL}}}{\alpha} m^2 \left( \log m + \frac{\pi}{2}\right), $$
with a fixed $0 < \alpha< 1$, and a constant $c^{\mathrm{CL}} \geq 1$ independent of $m$, $n$ and related to the Lebesgue constant $\Lambda_{m} (\mathcal{S}_{m}^{\mathrm{CL}})$ of the Chebyshev segments $\mathcal{S}_{m}^{\mathrm{CL}}$. 
Then, the Lebesgue constant $ \Lambda_{m} (\mathcal{S}_m^{\mathrm{MC}}) $ associated with the concatenated mock-Chebyshev segments $ \mathcal{S}_m^{\mathrm{MC}} = \{ s_1^{\mathrm{MC}}, \ldots, s_m^{\mathrm{MC}} \} $ grows at most logarithmically with the bound
\[ \Lambda_{m} (\mathcal{S}_{m}^{\mathrm{MC}}) \leq \frac{c^{\mathrm{CL}}}{1-\alpha} \left( \log m + \frac{\pi}{2}\right).\]
\end{theorem}

\begin{remark}
Numerical comparisons indicate that the constant $c^{\mathrm{CL}}$ in Theorem \ref{thm:stabilityconcatenated} can be bounded by $2$. This gets also visible in Fig. \ref{fig:LebCheb}. 
\end{remark}

\begin{proof}
The Lebesgue constant associated to the Chebyshev segments $ \mathcal{S}_{m}^{\mathrm{CL}}$ is bounded by
    \begin{equation} \label{eq:LebCL}
        \Lambda_{m} (\mathcal{S}_{m}^{\mathrm{CL}}) \leq c^{\mathrm{CL}} \left( \log m + \frac{\pi}{2}\right),
    \end{equation}
with a constant $ c^{\mathrm{CL}} > 0$ given as the uniform bound of a sequence of operator norms, see \cite[Corollary 5.7]{Bruno:2023:PIO}. We choose $n$ and $\varepsilon > 0$ such that
\[n \geq \frac{1}{ \varepsilon} = \frac{c^{\mathrm{CL}}}{\alpha} m^2 \left( \log m + \frac{\pi}{2}\right).\]
Then, Lemma \ref{lem:closesegments} guarantees the existence and uniqueness of the mock-Chebyshev segments $\mathcal{S}_{m}^{\mathrm{MC}}$ and the stability result of Theorem \ref{thm-main} provides the estimate
$$
\Lambda_m(\mathcal{S}_{m}^{\mathrm{MC}}) \leq \frac{1}{1-\alpha} \Lambda_m(\mathcal{S}_{m}^{\mathrm{CL}}) \leq \frac{c^{\mathrm{CL}}}{1-\alpha} \left( \log m + \frac{\pi}{2}\right)
$$
for the Lebesgue constant of the mock-Chebyshev nodes. \qed
\end{proof}

\begin{figure}
\centering
\includegraphics[width=.8\textwidth]{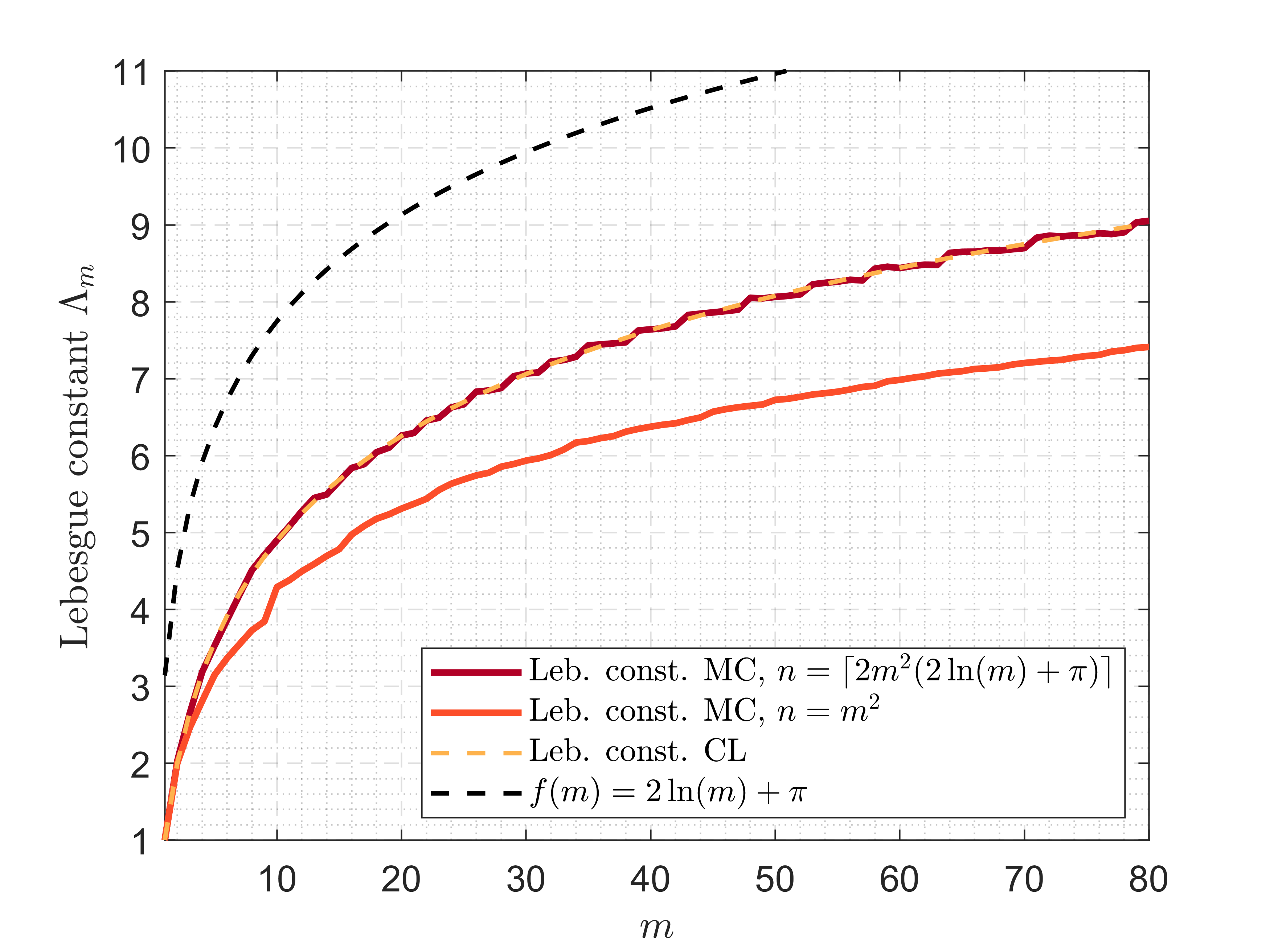}\hfill
    \caption{The Lebesgue constant $ \Lambda_m(\mathcal{S}_{m}^{\mathrm{MC}}) $ of the mock-Chebyshev segments $ \mathcal{S}_{m}^{\mathrm{MC}}$ depending on the selection of the equidistant grid size $n$. The Lebesgue constant $\Lambda_m(\mathcal{S}_{m}^{\mathrm{CL}})$ for the Chebyshev segments is highlighted in yellow}
 \label{fig:LebCheb} 
\end{figure}

{\bfseries Calculation of the segmental polynomial approximant.} Exploiting the linearity of the integral, the mean value of the function $f$ on the mock-Chebyshev segments $s_j^{\mathrm{MC}}$ can be calculated by averaging the integral information over the smaller equispaced segments $s_{\iota_j+1}, \dots,s_{\iota_j+k_j}$, as for instance illustrated in Fig.~\ref{prbseg}. By fixing a basis 
\begin{equation*}
    \mathcal{B}_{m}=\left\{u_1,\dots,u_{m}\right\}
\end{equation*}
of the polynomial space $\mathbb{P}_{m-1}$, the unique polynomial satisfying the histopolation conditions \eqref{eq:histopolatingconditions} can be written as
\begin{equation}\label{polPMC}
p^{\mathrm{MC}}_{m-1}(x)=p^{\mathrm{MC}}_{m-1}(f,x)=\sum_{i=1}^{m} a^{\mathrm{MC}}_i u_i(x),  
\end{equation}
where the vector $\boldsymbol{a}=\left[a^{\mathrm{MC}}_1,\dots,a^{\mathrm{MC}}_{m}\right]^T$ is the solution of the following linear system
  \begin{equation*}
\begin{bmatrix}
\mu^{\mathrm{MC}}_{1}(u_1) & \mu^{\mathrm{MC}}_{1}(u_2) & \cdots & \mu^{\mathrm{MC}}_{1}(u_{m})\\
\mu^{\mathrm{MC}}_{2}(u_1) & \mu^{\mathrm{MC}}_{2}(u_2) & \cdots & \mu^{\mathrm{MC}}_{2}(u_{m})\\
\vdots  & \vdots  & \ddots & \vdots  \\
\mu^{\mathrm{MC}}_{m}(u_1) & \mu^{\mathrm{MC}}_{m}(u_2) & \cdots & \mu^{\mathrm{MC}}_{m}(u_{m})\\
\end{bmatrix}
\begin{bmatrix}
a^{\mathrm{MC}}_1 \\
a^{\mathrm{MC}}_2\\
\vdots\\
a^{\mathrm{MC}}_{m}
\end{bmatrix}=
\begin{bmatrix}
\mu^{\mathrm{MC}}_{1}(f) \\
\mu^{\mathrm{MC}}_{2}(f)\\
\vdots\\
\mu^{\mathrm{MC}}_{m}(f)
\end{bmatrix}
\end{equation*}
and the averages $\mu^{\mathrm{MC}}_{j}$, $j=1,\dots,m$, are defined in~\eqref{relintegrals}. A suitable basis for histopolation is for instance given by the Chebyshev polynomials of the second kind \cite{Bruno:2023:PIO}.


 \begin{figure}
\centering
\includegraphics[scale=0.45]{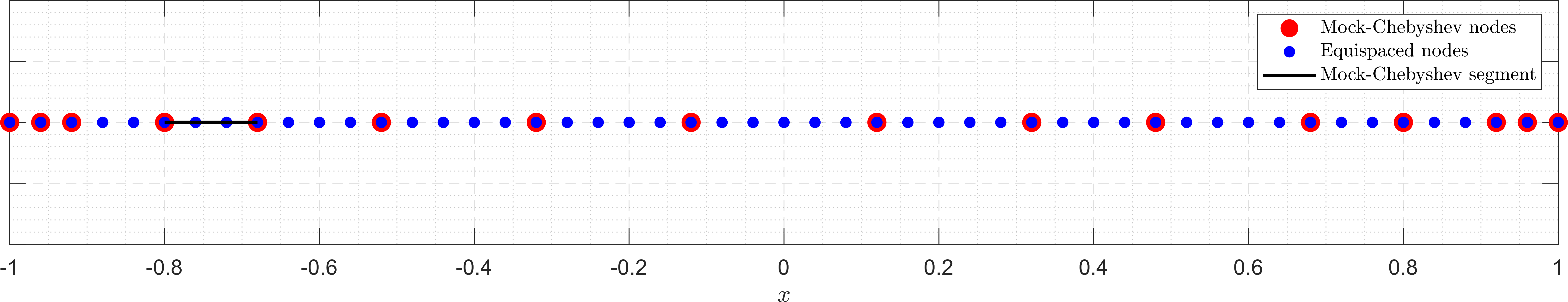}
\caption{Cartesian grid of $n+1=51$ equispaced nodes (\textcolor{blue}{$\bullet$}), mock-Chebyshev nodes ${X}_{m+1}^{\mathrm{MC}}$ of size $m+1=16$ (\textcolor{red}{$\bullet$}), the black line highlights the mock-Chebyshev segment $s_4^{\mathrm{MC}}=[x_3^{\mathrm{MC}},x_4^{\mathrm{MC}}]$. The mock-Chebyshev segments $\mathcal{S}_{m}^{\mathrm{MC}}$ partition the interval $I$ and have the nodes ${X}_{m+1}^{\mathrm{MC}}$ as end-points. }\label{prbseg}
\end{figure}

\subsection{Quasi-nodal mock-Chebyshev method} \label{sect:insidesegments}

The integral information of the function $f$ on the equispaced segements $\mathcal{S}_{n}^{\mathrm{eq}}$ can be reduced in different ways to approximate $f$ with a polynomial of degree $m-1$. In the first variant discussed above integral information was averaged in order to get approximate mean values of the function $f$ on Chebyshev segments. In the second variant introduced now, the uniform set $\mathcal{S}_{n}^{\mathrm{eq}}$ is reduced to a subset of size $m$ in such a way, that only those segments are taken that contain a root of the Chebyshev polynomial of the first kind of degree $m$, as visualized in Fig.~\ref{figgridreg}. If the number $n$ of segments gets large and the respective uniform length correspondingly small, this segmental technique mimics polynomial interpolation on the Chebyshev nodes, and we therefore refer to this technique as quasi-nodal mock-Chebyshev approximation.  

More precisely, we consider for this technique the Chebyshev nodes, i.e., the roots of the Chebyshev polynomials of the first kind of degree $m$ given as
\begin{equation*}
    X_m^{\mathrm{CF}} = \left\{ x^{\mathrm{CF}}_1, \ldots, x^{\mathrm{CF}}_m \right\}, \qquad  x_i^{\mathrm{CF}}=\cos\left(\frac{2 i - 1}{2 m }\pi\right), \quad i=1,\dots,m.
\end{equation*}
The idea of this new method consists of interpolating the function $f$ only on the set of segments
\begin{equation*}
    \mathcal{S}_{m}^{\mathrm{MCF}}=\left\{s_{\iota_1},\dots s_{\iota_{m}}\right\},
\end{equation*}
such that every segment
\begin{equation}   \label{siota}
   s_{\iota_i}=[x_{\iota_{i}-1},x_{\iota_{i}}], \qquad i=1,\dots,m,
\end{equation}
contains a Chebyshev node
\begin{equation*}
    x_{i}^{\mathrm{CF}}\in s_{\iota_i}, \qquad i=1,\dots,m.
\end{equation*}
If a node $x_{i}^{\mathrm{CF}} = x_{\iota_{i}}$ lies on the border of two equispaced segments, we associate the left hand segment to the node $x_{i}^{\mathrm{CF}}$ such that $s_{\iota_i}=[x_{\iota_{i}-1},x_{\iota_{i}}]$. The next lemma gives us a sufficient condition on the number $n$ to guarantee the existence and uniqueness of the segments $\mathcal{S}_{m}^{\mathrm{MCF}}$. 

\begin{lemma} \label{lem:refiningsegm}
Let $\varepsilon > 0$, $ m \in \mathbb{N}$ be given, and $ X_m^{\mathrm{CF}} = \left\{ x_1^{\mathrm{CF}}, \ldots, x_m^{\mathrm{CF}} \right\} $ the roots of the Chebyshev polynomial of the first kind of degree $m$. 
Choosing $n \in \mathbb{N}$ such that $$ n \geq \max \left \{  m^2 \frac{8}{\pi^2} , \, \frac{2}{\epsilon}\right \} $$ we get a unique set $\mathcal{S}_{m}^{\mathrm{MCF}} =\left\{s_{\iota_1},\dots s_{\iota_{m}}\right\}$ of distinct uniform segments in $\mathcal{S}_n^{\mathrm{eq}}$ such that $x_{i}^{\mathrm{CF}} \in s_{\iota_i}$ and $ \left\lvert s_{\iota_i} \right\rvert \leq \varepsilon $ for all $ i=1,\dots,m$.
\end{lemma}

\begin{proof}
Let $X^{\mathrm{CL}}_{2m+1}$ be the Chebyshev-Lobatto points of order $2m + 1$. Then, the Chebyshev nodes $X^{\mathrm{CF}}_{m}$ are a subset of $X^{\mathrm{CL}}_{2m+1}$. By Lemma \ref{lem:closesegments} (which is based on the observations in \cite{Boyd:2009:DRP}), the choice $n \geq \max\{8 m^2/\pi^2, \frac{2}{\varepsilon} \}$ guarantees the uniqueness of mock-Chebyshev points
$x_i^{\mathrm{MC}} = x_{\iota_i}$, $i = 0,\ldots,2m$ with $|x_i^{\mathrm{MC}} - x_i^{\mathrm{CL}}| \leq \varepsilon/2$. Therefore, also for $x_i^{\mathrm{CF}} \in X^{\mathrm{CF}}_{m}$ we have a unique $x_i^{\mathrm{MC}} = x_{\iota_i}$ with the same property $|x_i^{\mathrm{MC}} - x_i^{\mathrm{CF}}| \leq \varepsilon/2$. Without loss of generality, we can assume that $x_i^{\mathrm{MC}} > x_i^{\mathrm{CF}}$. Then, the node $x_i^{\mathrm{CF}}$ is contained in the segment $s_{\iota_i} = [x_{\iota_i-1}, x_{\iota_i}]$ and $\left\lvert s_{\iota_i} \right\rvert \leq \varepsilon $ since $x_{\iota_i}$ is the closest element of $X_{n+1}^{\mathrm{eq}}$ to $x_i^{\mathrm{CF}}$. Further, no other element $X^{\mathrm{CF}}_{m}$ is contained in $s_{\iota_i}$, since, by the interlacing of the nodes in $X^{\mathrm{CF}}_{m}$ and $X^{\mathrm{CL}}_{2m+1} \setminus X^{\mathrm{CF}}_{m}$ the node $x_{\iota_i-1}$ cannot be the closest element of $X^{\mathrm{eq}}_{n+1}$ to an element in $X^{\mathrm{CF}}_{m}$. \qed
\end{proof}  

{\bfseries Calculation of the quasi-nodal mock-Chebyshev approximant. } By fixing a basis 
\begin{equation*}
    \mathcal{B}_m=\left\{u_1,\dots,u_{m}\right\}
\end{equation*}
of the polynomial space $\mathbb{P}_{m-1}$, the polynomial histopolant obtained by this new procedure can be expanded as
\begin{equation*}
p^{\mathrm{MCF}}_{m-1}(x)=p^{\mathrm{MCF}}_{m-1}(f,x)=\sum_{i=1}^{m} a^{\mathrm{MCF}}_i u_i(x),   
\end{equation*}
where $\boldsymbol{a}^{\mathrm{MCF}}=\left[a^{\mathrm{MCF}}_1,\dots,a^{\mathrm{MCF}}_{m}\right]^T$ is the solution of the following linear system
  \begin{equation*}
\begin{bmatrix}
\mu_{\iota_1}(u_1) & \mu_{\iota_1}(u_2) & \cdots & \mu_{\iota_1}(u_{m})\\
\mu_{\iota_2}(u_1) & \mu_{\iota_2}(u_2) & \cdots & \mu_{\iota_2}(u_{m})\\
\vdots  & \vdots  & \ddots & \vdots  \\
\mu_{\iota_{m+1}}(u_1) & \mu_{\iota_{m}}(u_2) & \cdots & \mu_{\iota_{m}}(u_{m})\\
\end{bmatrix}
\begin{bmatrix}
a^{\mathrm{MCF}}_1 \\
a^{\mathrm{MCF}}_2\\
\vdots\\
a^{\mathrm{MCF}}_{m}
\end{bmatrix}=
\begin{bmatrix}
\mu_{\iota_1}(f) \\
\mu_{\iota_2}(f)\\
\vdots\\
\mu_{\iota_{m}}(f)
\end{bmatrix}
\end{equation*}
and $\mu_i$, $i=1,\dots,n$, are defined in~\eqref{hypoth}.

\begin{remark}\label{rem1}
We observe that, similarly as in the previously described method, also this approach uses only a reduced set of the totally available data. Specifically in this second approach, there are $n-m-1$ points of the data set that remain completely neglected. Therefore, it is reasonable to consider also more advanced approaches, as described in the next section, to enhance the accuracy of the approximation method by exploiting this additional information.
\end{remark}

 \begin{figure}
\centering
\includegraphics[scale=0.45]{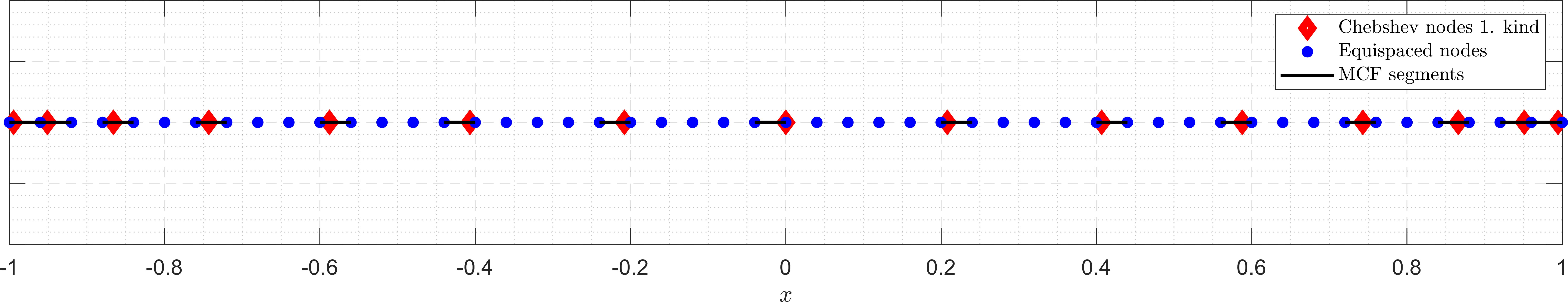}
\caption{Cartesian grid with $n+1=51$ equidistant nodes (\textcolor{blue}{$\bullet$}) and the Chebyshev nodes $X_{m}^{\mathrm{CF}}$ of first kind of order $m=16$ (\textcolor{red}{$\blacklozenge$}). The mock-Chebyshev segments $\mathcal{S}_{m}^{\mathrm{MCF}}$ are visualized with black lines and are characterized as elements of the uniform segments $\mathcal{S}_{n}^{\mathrm{eq}}$ that contain a Chebyshev node. }\label{figgridreg}
\end{figure}

With respect to the numerical conditioning of this second method, we get the following result for the Lebesgue constant $ \Lambda_m (\mathcal{S}_{m}^{\mathrm{MCF}})$ related to polynomial interpolation on the mock-Chebyshev segments $ \mathcal{S}_{m}^{\mathrm{MCF}} $ .

\begin{theorem} \label{thm:stabilityquasinodal}
 For $m \in \mathbb{N}$, let 
$$ n \geq \frac{2}{\alpha} m^2 \left( \frac{2}{\pi}\log m + 1\right), $$
with a fixed $0 < \alpha< 1$. Then, the Lebesgue constant $ \Lambda_{m} (\mathcal{S}_{m}^{\mathrm{MCF}}) $ associated with the mock-Chebyshev segments $ \mathcal{S}_{m}^{\mathrm{MCF}}$ grows at most logarithmically with the bound
\[ \Lambda_{m} (\mathcal{S}_{m}^{\mathrm{MCF}}) \leq \frac{1}{1-\alpha} \left( \frac{2}{\pi}\log m + 1 \right).\]
\end{theorem}

\begin{proof}
We consider the nodal Lebesgue constant $ \Lambda_{m} ({X}_m^{\mathrm{CF}}) $ associated with the Chebyshev nodes $ X_m^{\mathrm{CF}} = \left\{ x_1^{\mathrm{CF}}, \ldots, x_m^{\mathrm{CF}} \right\} $. It is well-known that this Lebesgue constant grows logarithmically in $ m $ and can be bounded as (see e.g. \cite[Sect. 3.2]{IbrahimogluSurvey})
\[\Lambda_{m} ({X}_m^{\mathrm{CF}}) \leq \frac{2}{\pi}\log m + 1.\]
We select now $\varepsilon > 0$ and $n \geq 2/\varepsilon$ such that
\[n \geq \frac{2}{\varepsilon} =  \frac{2}{\alpha} m^2 \left( \frac{2}{\pi} \log m + 1\right).\]
With this, Lemma \ref{lem:refiningsegm} guarantees the uniqueness of the mock-Chebyshev segments $\mathcal{S}_{m}^{\mathrm{MCF}} =\left\{s_{\iota_1},\dots s_{\iota_{m}}\right\}$ with the property that $x_{i}^{\mathrm{CF}} \in s_{\iota_i}$ and $ \left\lvert s_{\iota_i} \right\rvert \leq \varepsilon $ for all $ i=1,\dots,m$. Furthermore, the stability result of Theorem \ref{thm-main} gives the bound 
$$
\Lambda_m(\mathcal{S}_{m}^{\mathrm{MCF}}) \leq \frac{1}{1-\alpha} \Lambda_m(\mathcal{X}_m^{\mathrm{CF}}) \leq \frac{1}{1-\alpha} \left( \frac{2}{\pi}\log m + 1 \right).
$$
for the Lebesgue constant of the mock-Chebyshev segments. \qed
\end{proof}

\begin{remark}
In Theorem \ref{thm:stabilityquasinodal}, we observe that the larger the number $n$ of uniform segments is picked, meaning that also $\alpha$ can be chosen respectively small, the closer the bound of the Lebesgue constant of the quasi-nodal histopolation gets to the one of the Chebyshev nodes. This is certainly not surprising since in the limit $n \to \infty$ the quasi-nodal histopolant of a continuous function $f$ converges to the polynomial interpolant on the Chebyshev nodes ${X}_m^{\mathrm{CF}}$. 

In comparison to the condition $n \geq \frac{2}{\alpha} m^2 \left( \frac{2}{\pi}\log m + 1\right)$ of Theorem \ref{thm:stabilityquasinodal} which rigorously guarantees the stability of the Lebesgue constant for the mock-Chebyshev segments, many works related to mock-Chebyshev approximation typically use a quadratic relation $n \sim m^2$ between the two main parameters $m$ and $n$.  
In fact, the numerical tests displayed in Fig. \ref{fig:LebConstSMCF} indicate that the asymptotic logarithmic behavior of the Lebesgue constant $ \Lambda_m (\mathcal{S}_{m}^{\mathrm{MCF}}) $ might hold true already when $ n = m^2$. 
\end{remark}

\begin{figure}[htbp]
    \centering
    \includegraphics[width=.8\textwidth]{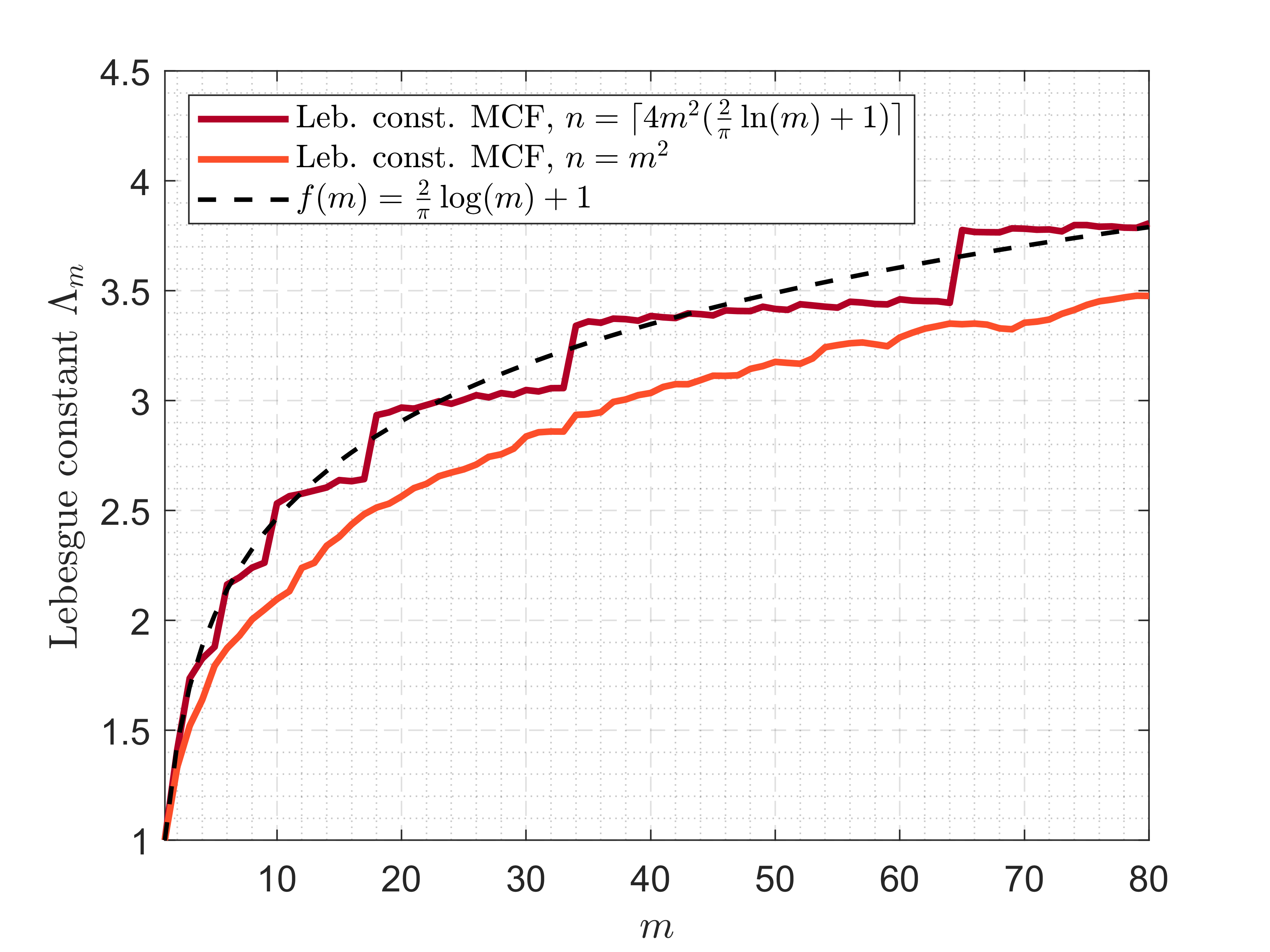}
    \caption{Segmental Lebesgue constant $\Lambda_m (\mathcal{S}_{m}^{\mathrm{MCF}})$ for the quasi-nodal segments $\mathcal{S}_{m}^{\mathrm{MCF}}$ depending on the selection of the equidistant grid size $n$.}
    \label{fig:LebConstSMCF}
\end{figure}

\subsection{Constrained segmental mock-Chebyshev method}
This last method aims to improve the approximation accuracy obtained by the quasi-nodal mock-Chebyshev segments introduced in the last section. As noted in Remark~\ref{rem1}, the previous method uses only $m$ pieces of data out of the $n$ available equispaced segments, neglecting the remaining $n-m$ integral values. We now seek to leverage these unused data to enhance the accuracy when approximating a smooth function. Drawing inspiration from the definition of the constrained mock-Chebyshev least-squares approximation in classical nodal interpolation, the idea is to exploit the $n-m$ unused data for a simultaneous regression. We combine the action of histopolation and least-squares approximation by enforcing area-matching conditions on a fixed set of quasi-nodal segments and a least-squares deviation for the remaining segments. This improves the overall quality of the approximant, as numerical experiment will show. An important aspect is the choice of the total degree $ r \geq m $. In analogy to the definition of the constrained mock-Chebyshev least-squares approximation in the classical theory, we set
 \begin{equation*}
     r=m+\left\lfloor \pi\sqrt{\frac{n}{12}} \right\rfloor+1,
 \end{equation*}
where $m$ is defined in~\eqref{valueofm}. This choice of the degree $r$ leads to a good approximation in the uniform norm~\cite{DeMarchi:2015:OTC}. 
We consider a basis  
\begin{equation*}
\mathcal{B}_{r}=\{u_1,\dots,u_r\}   
\end{equation*}
 of the polynomial space $\mathbb{P}_{r-1}$ such that
\begin{equation*}
    \operatorname{span}\left\{u_1,\dots,u_r\right\}=\mathbb{P}_{r-1}. 
\end{equation*}
In the interest of simplicity and concise notation, we consider the $m$ segments
\begin{equation*}
    s_{i}=s_{\iota_i}, \qquad i=1,\dots,m,
\end{equation*}
where the equispaced quasi-nodal segments $s_{\iota_i}$ have been defined in~\eqref{siota}. 
Consistently with this notion, we consider then the integral data
\begin{equation*}
    \mu_i(f)=\mu_{\iota_i}(f), \qquad i=1,\dots,m.
\end{equation*}
and the interpolation matrix
\begin{equation}\label{interpmatrix}
N_{n,r}=\begin{bmatrix}
\mu_1(u_1) & \mu_1(u_2) & \cdots & \mu_1(u_r)\\
\mu_2(u_1) & \mu_2(u_2) & \cdots & \mu_2(u_r)\\
\vdots  & \vdots  & \ddots & \vdots  \\
\mu_{n}(u_1) & \mu_{n}(u_2) & \cdots & \mu_{n}(u_r)\\
\end{bmatrix},
\end{equation}
as well as the submatrix $N_{m,r}$, formed by taking the first $m$ rows of $N_{n,r}$. We further use the abbreviations 
\begin{equation*}
\boldsymbol{b}=\left[  \mu_1(f),\dots,\mu_{n}(f)\right]^T, \qquad     \boldsymbol{d}=\left[\mu_1(f),\dots,\mu_{m}(f)\right]^T.
\end{equation*}
Then, the polynomial approximant obtained by this new method can be written as
\begin{equation}\label{ConMC}
    \hat{p}^{\mathrm{MCF}}_r(x)=\hat{p}^{\mathrm{MCF}}_r(f,x)=\sum_{i=1}^r \hat{a}_i^{\mathrm{MCF}}u_i(x),
\end{equation}
where $\boldsymbol{\hat{a}}^{\mathrm{MCF}}=\left[\hat{a}_1^{\mathrm{MCF}},\dots, \hat{a}_r^{\mathrm{MCF}}\right]^T$
is the solution of the linear system
\begin{equation}
    \begin{bmatrix}
G_{r,r} &  N_{m,r}^T  \\
 N_{m,r} & 0  \\
\end{bmatrix}
\begin{bmatrix}
\boldsymbol{\hat{a}}^{\mathrm{MCF}} \\
{\boldsymbol{z}}^{\mathrm{MCF}} \\
\end{bmatrix}=
\begin{bmatrix}
 \boldsymbol{c}\\
\boldsymbol{d} \\
\end{bmatrix}, \qquad
\end{equation}
 ${\boldsymbol{z}}^{\mathrm{MCF}}$ is a vector of Lagrange multipliers and 
\begin{equation}
\label{defW}
G_{r,r}=2 N_{n,r}^T N_{n,r}, \qquad \boldsymbol{c}=2  N_{n,r}^T\boldsymbol{b}.
\end{equation}
The matrix 
\begin{equation}\label{KKTmatrix}
    M=    \begin{bmatrix}
G_{r,r} &  N_{m,r}^T  \\
 N_{m,r} & 0  \\
\end{bmatrix}
\end{equation}
is called KKT matrix. In order to prove that the approximating polynomial~\eqref{ConMC} is well-defined, the KKT matrix $M$ has to be invertible. This is indeed true due to the fact that the $n$ functionals $\mu_i$ are linearly independent in the space $\mathbb{P}_{n-1}$.  
\begin{theorem}
    The matrix $M$ defined in~\eqref{KKTmatrix} is nonsingular. 
\end{theorem}
\begin{proof}
To demonstrate the validity of this theorem, we have to prove that both, the matrix  $N_{n,r}$ defined in~\eqref{interpmatrix}, as well as its submatrix $N_{m,r}$ have full rank~\cite{DellAccio:2022:GOT}. By extending the basis $\mathcal{B}_r=\{u_1,\dots, u_r\}$ to a basis 
\begin{equation*}
    \mathcal{B}_{n}=\{u_1,\dots,u_r, u_{r+1}, \dots, u_{n}\}
\end{equation*} 
of the vector space $\mathbb{P}_{n-1}$, we can construct the full interpolation matrix 
  \begin{equation}\label{prova}
   N_{n,n}=\begin{bmatrix}
\mu_{1}(u_1) & \mu_{1}(u_2) & \cdots & \mu_{1}(u_{n})\\
\mu_{2}(u_1) & \mu_{2}(u_2) & \cdots & \mu_{2}(u_{n})\\
\vdots  & \vdots  & \ddots & \vdots  \\
\mu_{n}(u_1) & \mu_{n}(u_2) & \cdots & \mu_{n}(u_{n})\\
\end{bmatrix}.
\end{equation}
Since the equispaced segments form a connected chain of intervals, this matrix is nonsingular~\cite[Prop. 3.1]{Bruno:2023:PIO}, meaning its columns are linearly independent.
Since $N_{n,r}$ is a submatrix of~\eqref{prova} consisting of its first $r$ columns, also the columns of $N_{n,r}$  are linearly independent. Therefore, since $r<n$, the matrix $N_{n,r}$ has full rank.

It remains to prove that the matrix $N_{m,r}$ is of full rank.
To this aim, we consider the matrix
     \begin{equation}\label{pr1}
   N_{m,m}=\begin{bmatrix}
\mu_1(u_1) & \mu_1(u_2) & \cdots & \mu_1(u_m)\\
\mu_2(u_1) & \mu_2(u_2) & \cdots & \mu_2(u_m)\\
\vdots  & \vdots  & \ddots & \vdots  \\
\mu_{m}(u_1) & \mu_{m}(u_2) & \cdots & \mu_{m}(u_m)\\
\end{bmatrix}.
\end{equation}
Since the matrix~\eqref{pr1} is known to be nonsingular~\cite[Prop. 3.1]{Bruno:2023:PIO}, its rows are linearly independent. As $N_{m,m}$ is a submatrix of $N_{m,r}$, also the rows of $N_{m,r}$ are linearly independent. Thus, since $r > m$, also the matrix $N_{m,r}$ has maximum rank. \qed
\end{proof}


\section{Numerical experiments}
\label{SecNum}
In this section, we numerically test the accuracy of the proposed methods by several examples. We consider the test functions
\begin{equation} \label{testfun}
    f_1(x)=\frac{1}{1+25x^2}, \quad f_2(x)=\frac{1}{1+8x^2}, \quad f_3(x)=e^{x^2+1}, 
\end{equation}
\begin{equation} \label{testfun1}
   f_4(x)=\cos(5x), \quad f_5(x)=\frac{1}{x-1.5}, \quad f_6(x)=x\left\lvert x\right\rvert^3.
\end{equation}
In all experiments, we use the Chebyshev polynomial of the first kind given by
\begin{equation*}
    T_0(x)=1, \quad T_1(x)=x, \quad T_{k+1}(x)=2x T_k(x)-T_{k-1}(x), \quad k\ge 2,
\end{equation*}
as basis polynomials for the respective spaces. 
We conduct two types of numerical experiments. In the first type, we compare the maximum approximation errors produced by approximating the functions $f_1, \ldots, f_6$  using polynomial interpolation with those produced by applying our methods on a uniform grid with $n=50$ segments
\begin{equation*}
    e_n^{\mathrm{eq}}=\max_{x\in [0,1]} \left\lvert f(x)-p^{\mathrm{eq}}_{n-1}(x) \right\rvert, \qquad e_{m}^{\mathrm{MC}}=\max_{x\in [0,1]} \left\lvert f(x)-p^{\mathrm{MC}}_{m-1}(x) \right\rvert,
\end{equation*}
\begin{equation*}
 {e}_{m}^{\mathrm{MCF}}=\max_{x\in [0,1]} \left\lvert f(x)-{p}^{\mathrm{MCF}}_{m-1}(x) \right\rvert,  \qquad \hat{e}_{r}^{\mathrm{MCF}}=\max_{x\in [0,1]} \left\lvert f(x)-\hat{p}^{\mathrm{MCF}}_{r-1}(x) \right\rvert.
\end{equation*}
 The results of this first test are summarized in Table~\ref{tab:firsttesti}.  From these results, we observe a notable improvement in approximation accuracy produced by our methods. Specifically, we note that the error produced by the interpolation on equally spaced segments based on $X_{n+1}^{\mathrm{eq}}$, even with $n=50$, is not comparable to that produced by our approximation methods. This discrepancy can be attributed to the condition number of the relative Vandermonde matrix and KKT matrix in the case of the constrained segmental mock-Chebyshev method. In fact, the condition number of the Vandermonde matrix associated to interpolation on equispaced segments grows dramatically even for small values of $n$. In contrast, the condition numbers of the Vandermonde matrices associated to our methods and that of the KKT matrix in the constrained segmental mock-Chebyshev method remain low even for large $n$, as shown in Fig.~\ref{CondNumbtrend}.

 \begin{table}
    \centering
    \begin{tabular}{ c | c | c |c|c } 
 & $e^{\mathrm{eq}}_{n}$ & $e^{\mathrm{MC}}_{m}$ & $e^{\mathrm{MCF}}_m$ & $\hat{e}_r^{\mathrm{MCF}}$\\ \hline \hline
        $f_1(x)$ &  4.77e+06  &  6.19e-02 & 7.39e-02& 2.67e-01\\ \hline
        $f_2(x)$ & 6.57e+02 &  1.12e-02 & 9.19e-03  & 1.25e-02\\ \hline
        $f_3(x)$ & 4.05e-03 &  2.10e-08& 8.48e-10 & 5.90e-13\\ \hline
        $f_4(x)$ & 2.67e-03 & 9.12e-07 & 2.85e-08 & 7.43e-13\\ \hline
        $f_5(x)$ & 1.68e-03 &  6.43e-06& 1.61e-06  & 2.94e-08\\ \hline
        $f_6(x)$ & 1.53e+06 & 1.31e-04 & 1.22e-04 & 2.33e-04\\ \hline
    \end{tabular}
    \caption{Comparison between the approximation accuracy produced by the interpolation of equispaced segments ($e^{\mathrm{eq}}_n$) with that obtained through the concatenated mock-Chebyshev method ($e^{\mathrm{MC}}_{m}$), the quasi-nodal mock-Chebyshev method  ($e_m^{\mathrm{MCF}}$) and the constrained mock-Chebyshev segments method ($\hat{e}_r^{\mathrm{MCF}}$) based on a grid of $n+1=51$ uniform nodes.}
    \label{tab:firsttesti}
\end{table}

In the second set of experiments, we analyze the behavior of the maximum approximation errors generated by approximating the test functions $f_1$-$f_6$ using our methods. This analysis is conducted using equally spaced grids $X_{n+1}^{\mathrm{eq}}$, with $n$ ranging from $50$ to $1000$.
The results of this experiment are shown in Fig.~\ref{fig1} and Fig.~\ref{fig2}.

\begin{figure}
\centering
\includegraphics[scale=0.5]{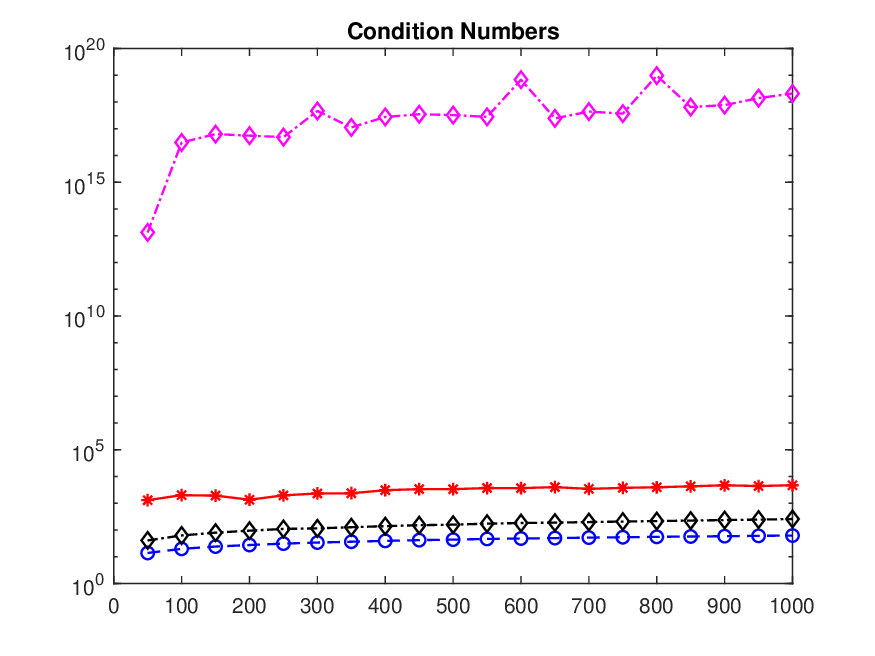}
\caption{Trend of the condition number of the Vandermonde matrix associated by the interpolation on equispaced segments (magenta line) with that obtained through the concatenated mock-Chebyshev method (blue line), the quasi-nodal mock-Chebyshev method  (black line) and the trend of the KKT matrix relative to the constrained mock-Chebyshev method (red line) based on a grid of $n+1=51:50:1001$ uniform nodes. }\label{CondNumbtrend}
\end{figure}

\begin{figure}
\centering
\includegraphics[width=.32\textwidth]{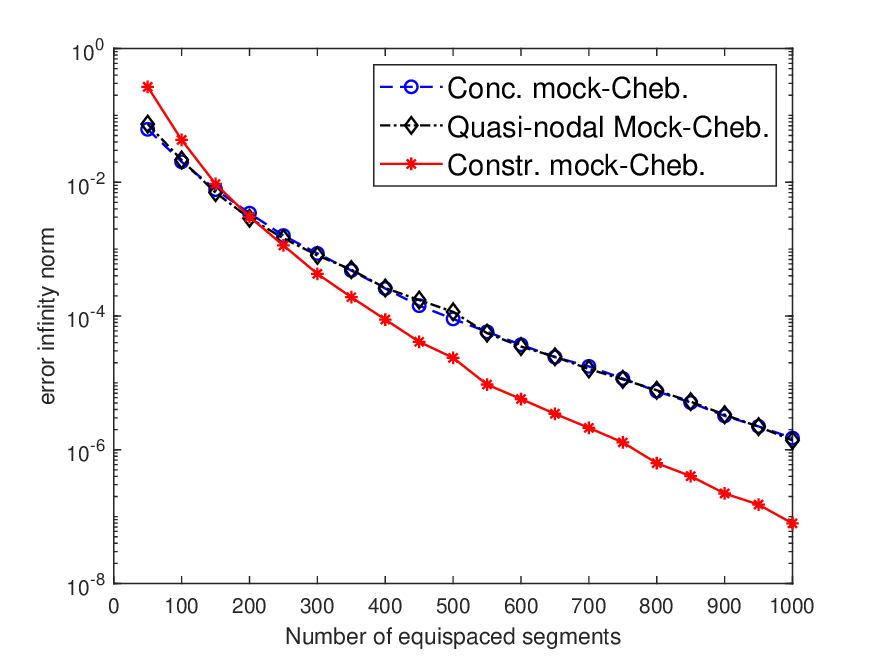}\hfil
\includegraphics[width=.32\textwidth]{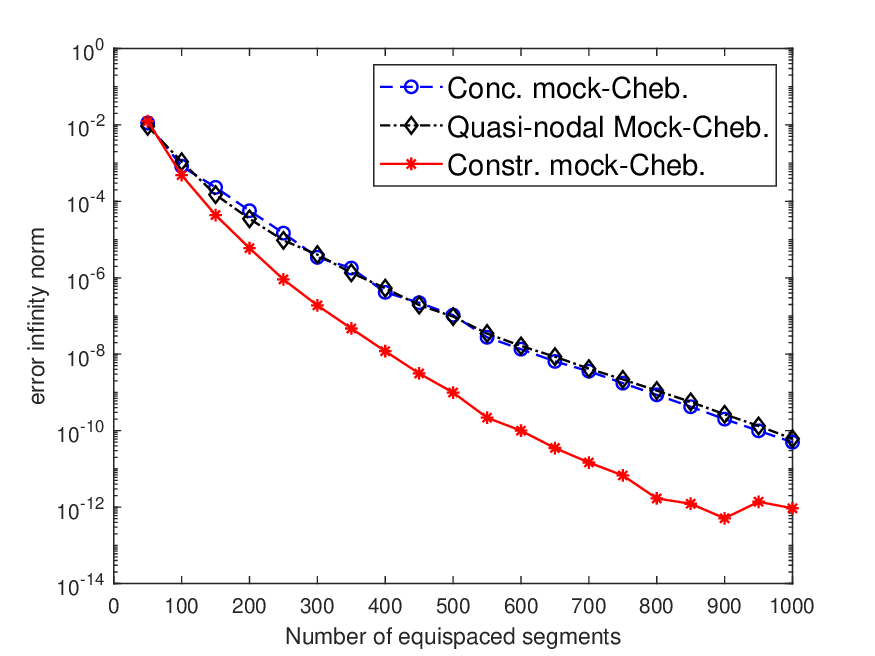}\hfil
\includegraphics[width=.32\textwidth]{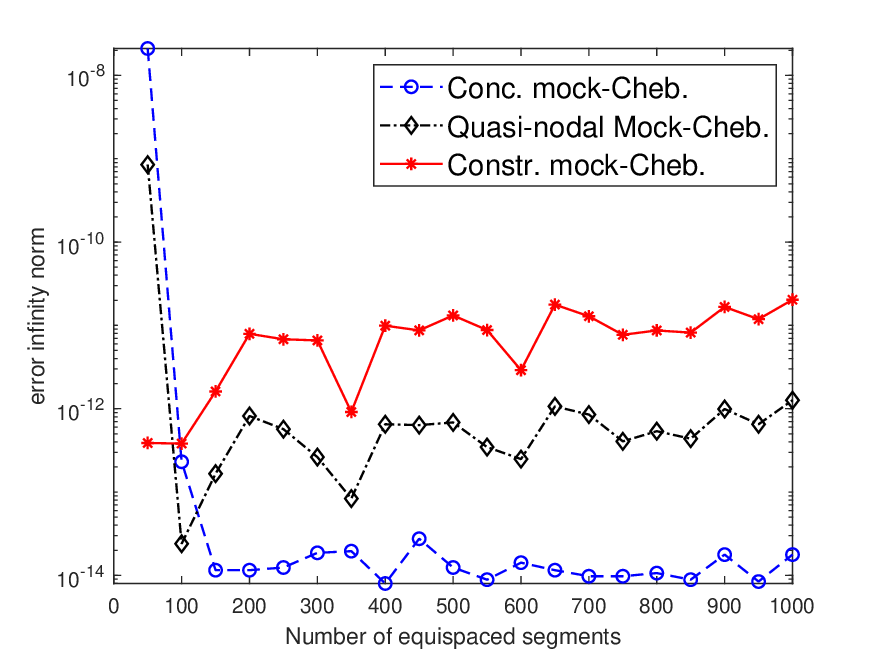}
    \caption{Trend of the maximum approximation error produced by approximating the functions $f_1$ (left), for $f_2$ (center) and $f_3$ (right) by using the concatenated mock-Chebyshev method (blue line), the quasi-nodal mock-Chebyshev method  (black line) and the constrained mock-Chebyshev method (red line) based on a grid of $n+1=51:50:1001$ equispaced nodes.}
 \label{fig1} 
\end{figure}

\begin{figure}
\centering
\includegraphics[width=.32\textwidth]{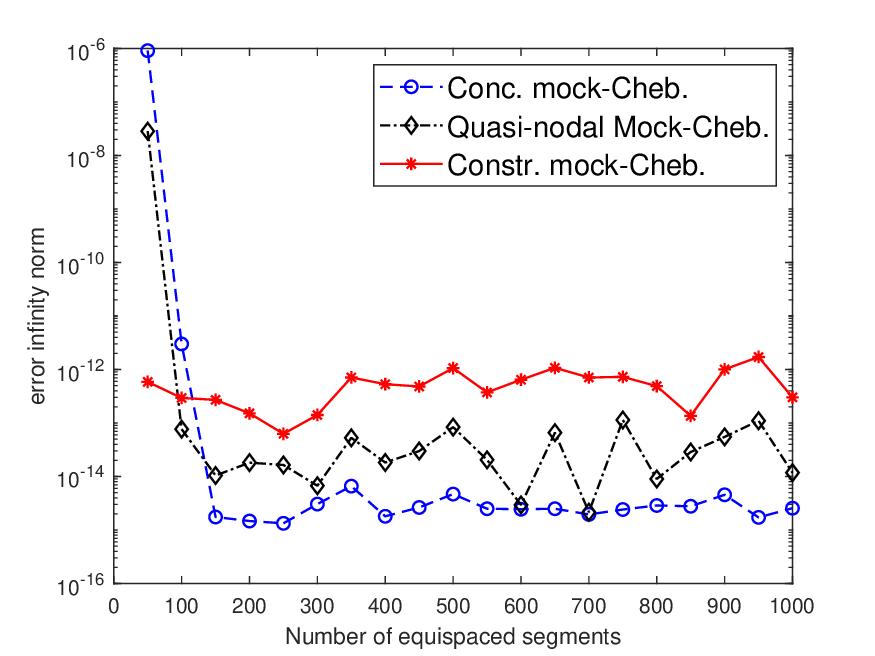}\hfil
\includegraphics[width=.32\textwidth]{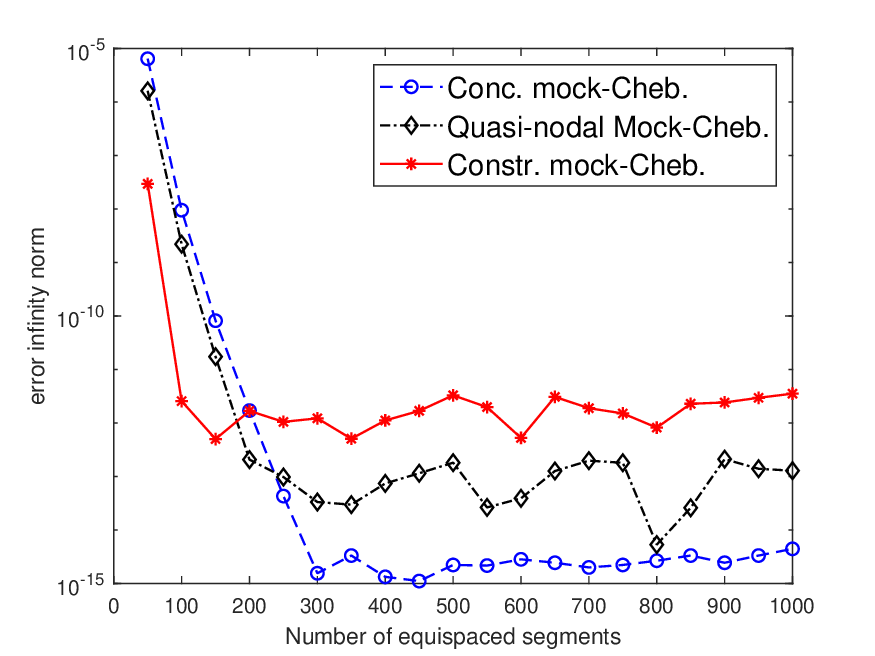}\hfil
\includegraphics[width=.32\textwidth]{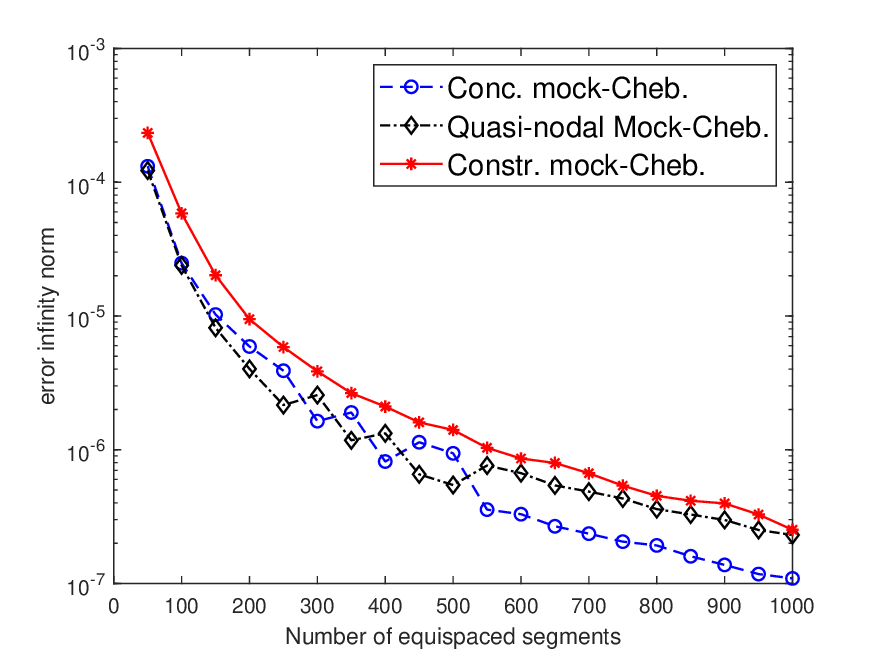}
    \caption{Trend of the maximum approximation error produced by approximating the functions $f_4$ (left), for $f_5$ (center) and $f_6$ (right) by using the concatenated mock-Chebyshev method (blue line), the quasi-nodal mock-Chebyshev method  (black line) and the constrained mock-Chebyshev method (red line) based on a grid of $n+1=51:50:1001$ equispaced nodes.}
 \label{fig2} 
\end{figure}
As can be seen in Fig. \ref{fig2}, the error trend generally decreases as $n$ increases. Once a maximum precision has been reached, increasing the number of nodes does not lead to a more accurate approximation, but rather remains constant. Furthermore, we observe that the regression technique in the constrained segmental mock-Chebyshev method does not always improve the accuracy of the approximation compared to the quasi-nodal mock-Chebyshev method.

\section*{Acknowledgments}
 This research has been conducted as part of RITA \textquotedblleft Research
 ITalian network on Approximation'' and as part of the UMI group \enquote{Teoria dell'Approssimazione
 e Applicazioni}. The research was supported by GNCS-INdAM 2024 projects. The first author is funded by IN$\delta$AM and supported by the University of Padova. The third and the fourth author are funded by the European Union – NextGenerationEU under the National Recovery and Resilience Plan (NRRP), Mission 4 Component 2 Investment 1.1 - Call PRIN 2022 No. 104 of February 2, 2022 of Italian Ministry of University and Research; Project 2022FHCNY3 (subject area: PE - Physical Sciences and Engineering) \enquote{Computational mEthods for Medical Imaging (CEMI)}.


\begin{thebibliography}{10}
\providecommand{\url}[1]{{#1}}
\providecommand{\urlprefix}{URL }
\expandafter\ifx\csname urlstyle\endcsname\relax
  \providecommand{\doi}[1]{DOI~\discretionary{}{}{}#1}\else
  \providecommand{\doi}{DOI~\discretionary{}{}{}\begingroup
  \urlstyle{rm}\Url}\fi

\bibitem{ABR22}
Alonso~Rodr\'{\i}guez, A., Bruni~Bruno, L., Rapetti, F.: Towards nonuniform
  distributions of unisolvent weights for high-order {W}hitney edge elements.
\newblock Calcolo \textbf{59}(4), Paper No. 37, 29 (2022)

\bibitem{BruniRunge}
Alonso~Rodr\'{\i}guez, A., Bruni~Bruno, L., Rapetti, F.: Whitney edge elements
  and the {R}unge phenomenon.
\newblock J. Comput. Appl. Math. \textbf{427}, Paper No. 115117, 9 (2023).
\newblock \doi{10.1016/j.cam.2023.115117}.
\newblock \urlprefix\url{https://doi.org/10.1016/j.cam.2023.115117}

\bibitem{ARRLeb}
Alonso~Rodr\'{\i}guez, A., Rapetti, F.: On a generalization of the {L}ebesgue's
  constant.
\newblock J. Comput. Phys. \textbf{428}, 109964 (2021)

\bibitem{Boyd:2009:DRP}
Boyd, J.P., Xu, F.: Divergence ({R}unge phenomenon) for least-squares
  polynomial approximation on an equispaced grid and {M}ock--{C}hebyshev subset
  interpolation.
\newblock Appl. Math. Comput. \textbf{210}, 158--168 (2009)

\bibitem{BruniErbFekete}
Bruni~Bruno, L., Erb, W.: The {F}ekete problem in segmental polynomial
  interpolation.
\newblock preprint  (2024).
\newblock Https://arxiv.org/pdf/2403.09378

\bibitem{Bruno:2023:PIO}
Bruni~Bruno, L., Erb, W.: Polynomial interpolation of function averages on
  interval segments.
\newblock to appear in SIAM J. Numer. Anal.  (2024)

\bibitem{Davis:1975:IAA}
Davis, P.J.: Interpolation and {A}pproximation.
\newblock Dover Publications, Inc., New York (1975)

\bibitem{DeMarchi:2015:OTC}
De~Marchi, S., Dell’Accio, F., Mazza, M.: On the constrained mock-{C}hebyshev
  least-squares.
\newblock J. Comput. Appl. Math. \textbf{280}, 94--109 (2015)

\bibitem{DeMarchi:2020:PIV}
De~Marchi, S., Marchetti, F., Perracchione, E., Poggiali, D.: Polynomial
  interpolation via mapped bases without resampling.
\newblock J. Comput. Appl. Math. \textbf{364}, 112347 (2020)

\bibitem{DellAccio:2022:AAA}
Dell'Accio, F., Di~Tommaso, F., Francomano, E., Nudo, F.: An adaptive algorithm
  for determining the optimal degree of regression in constrained
  mock-{C}hebyshev least squares quadrature.
\newblock Dolomites Res. Notes Approx. \textbf{15}, 35--44 (2022)

\bibitem{DellAccio:2022:CMC}
Dell’Accio, F., Di~Tommaso, F., Nudo, F.: Constrained mock-{C}hebyshev least
  squares quadrature.
\newblock Appl. Math. Lett. \textbf{134}, 108328 (2022)

\bibitem{DellAccio:2022:GOT}
Dell’Accio, F., Di~Tommaso, F., Nudo, F.: Generalizations of the constrained
  mock-{C}hebyshev least squares in two variables: {T}ensor product vs total
  degree polynomial interpolation.
\newblock Appl. Math. Lett. \textbf{125}, 107732 (2022)

\bibitem{DellAccio:2024:AEO}
Dell’Accio, F., Marcellán, F., Nudo, F.: An extension of a mixed
  interpolation-regression method using zeros of orthogonal polynomials.
\newblock J. Comput. Appl. Math. \textbf{450}, 116010 (2024)

\bibitem{DellAccio:2023:PIR}
Dell’Accio, F., Mezzanotte, D., Nudo, F., Occorsio, D.: Product integration
  rules by the constrained mock-{C}hebyshev least squares operator.
\newblock BIT Numer. Math. \textbf{63}, 24 (2023)

\bibitem{DellAccio:2024:NAO}
Dell’Accio, F., Mezzanotte, D., Nudo, F., Occorsio, D.: Numerical
  approximation of {F}redholm integral equation by the constrained
  mock-{C}hebyshev least squares operator.
\newblock J. Comput. Appl. Math. \textbf{447}, 115886 (2024)

\bibitem{DellAccio:2024:PAO}
Dell’Accio, F., Nudo, F.: Polynomial approximation of derivatives through a
  regression-interpolation method.
\newblock Appl. Math. Lett. \textbf{152}, 109010 (2024)

\bibitem{HiemstraJCP}
Hiemstra, R., Toshniwal, D., Huijsmans, R., Gerritsma, M.: High order geometric
  methods with exact conservation properties.
\newblock Journal of Computational Physics \textbf{257}, 1444--1471 (2014)

\bibitem{HiptmairXu}
Hiptmair, R., Xu, J.: Nodal auxiliary space preconditioning in {${\bf H}({\bf
  curl})$} and {${\bf H}({\rm div})$} spaces.
\newblock SIAM J. Numer. Anal. \textbf{45}(6), 2483--2509 (2007)

\bibitem{IbrahimogluSurvey}
Ibrahimoglu, B.A.: Lebesgue functions and {L}ebesgue constants in polynomial
  interpolation.
\newblock J. Inequal. Appl. pp. 1--15 (2016)

\bibitem{Ibrahimoglu:2020:AFA}
Ibrahimoglu, B.A.: A fast algorithm for computing the mock-{C}hebyshev nodes.
\newblock J. Comput. Appl. Math. \textbf{373}, 112336 (2020)

\bibitem{Ibrahimoglu:2024:ANF}
Ibrahimoglu, B.A.: A new fast algorithm for computing the mock-{C}hebyshev
  nodes.
\newblock Appl. Numer. Math.  (2024)

\bibitem{PiazzonVianello}
Piazzon, F., Vianello, M.: Stability inequalities for {L}ebesgue constants via
  {M}arkov-like inequalities.
\newblock Dolomites Res. Notes Approx. \textbf{11}, 1--9 (2018)

\bibitem{Rivlin}
Rivlin, T.J.: An introduction to the approximation of functions.
\newblock Dover Publications, Inc., New York (1981)

\bibitem{Robidoux}
Robidoux, N.: Polynomial histopolation, superconvergent degrees of freedom, and
  pseudospectral discrete hodge operators (2006)

\bibitem{Runge1901}
Runge, C.: {\"U}ber empirische {Funktionen} und die {Interpolation} zwischen
  {\"a}quidistanten {Ordinaten}.
\newblock Schl{\"o}milch Z. \textbf{46}, 224--243 (1901)

\bibitem{Schoenberg}
Schoenberg, I.J.: Splines and histograms.
\newblock In: Spline functions and approximation theory ({P}roc. {S}ympos.,
  {U}niv. {A}lberta, {E}dmonton, {A}lta., 1972), Internat. Ser. Numer. Math.,
  Vol. 21, pp. 277--327. Birkh\"{a}user Verlag, Basel-Stuttgart (1973)

\end{thebibliography}
 \end{document}